\documentclass[a4paper,UKenglish,cleveref, autoref, thm-restate]{lipics-v2021}

\pdfoutput=1 
\hideLIPIcs  

\graphicspath{
	{./}
	{CrushFigures/}
}

\newcommand{\tri}{\mathcal{T}}
\newcommand{\manifold}{\mathcal{M}}

\newtheorem{question}[theorem]{Question}
\newtheorem{algorithm}[theorem]{Algorithm}
\newtheorem{problem}[theorem]{Problem}

\usepackage{complexity}
\usepackage{tikz}
\usetikzlibrary{calc,arrows.meta}
\usepackage[colorinlistoftodos]{todonotes}
\usepackage{hyperref}

\Crefname{question}{Question}{Questions}
\Crefname{equestion}{Question}{Questions}
\crefalias{stepi}{enumi}
\Crefname{stepi}{Step}{Steps}
\crefalias{propi}{enumi}
\Crefname{propi}{Property}{Properties}
\crefalias{casei}{enumi}
\Crefname{casei}{Case}{Cases}
\Crefname{problem}{Problem}{Problems}

\creflabelformat{propi}{#2(#1)#3}

\title{A Practical Algorithm for Knot Factorisation}

\author{Alexander He}{401 Mathematical Sciences, Oklahoma State University, Stillwater, OK 74078, United States \and \url{https://sites.google.com/view/alex-he} }{alex.he@okstate.edu}{https://orcid.org/0000-0002-2189-4942}{Supported by an Australian Government Research Training Program Scholarship for part of this project}

\author{Eric Sedgwick}{School of Computing, DePaul University, 243 S Wabash Ave, Chicago, IL 60604, United States}{ESedgwick@cdm.depaul.edu}{https://orcid.org/0000-0003-1995-9097}{}

\author{Jonathan Spreer}{School of Mathematics and Statistics, The University of Sydney, Carslaw Building F07, Camperdown NSW 2006, Australia}{jonathan.spreer@sydney.edu.au}{https://orcid.org/0000-0001-6865-9483}{Supported by the Australian Research Council under the Discovery Project scheme (grant number DP220102588)}

\authorrunning{A.~He and E.~Sedgwick and J.~Spreer} 

\Copyright{Alexander He and Eric Sedgwick and Jonathan Spreer} 



\category{} 

\relatedversion{} 

\supplement{\textit{Software (Source Code)}: \url{https://github.com/AlexHe98/idealedge}}

\acknowledgements{This work was inspired by discussions at the Dagstuhl workshop
``Triangulations in Geometry and Topology''
\cite{DBLP:journals/dagstuhl-reports/BuchinCMSH24}. The authors would like to thank the
participants for interesting conversations about knots and normal surfaces.
Moreover, the authors thank Nathan Dunfield for suggesting a way to generate hard diagrams
of composite knots on which to test our algorithm. We also thank the anonymous reviewers
for their helpful comments.}

\nolinenumbers 

\begin{document}

\maketitle

\begin{abstract}
We present an algorithm for computing the prime factorisation of a knot, which is practical in the following sense:
using Regina, we give an implementation that works well for inputs of reasonable size, including prime knots from the $19$-crossing census.
The main new ingredient in this work is an object that we call an ``edge-ideal triangulation'', which is what our algorithm uses to represent knots.
As other applications, we give an alternative proof that prime knot recognition is in \coNP, and present some new complexity results for triangulations.
Beyond knots, our work showcases edge-ideal triangulations as a tool for potential applications in $3$-manifold topology.
\end{abstract}

\section{Introduction}\label{sec:intro}

Computational problems in $3$-manifold topology often exhibit a gap between theory and practice.
For example, consider the \textsc{Unknot Recognition} problem, which takes a knot $K$ as input, and asks whether $K$ is equivalent to the unknot (see \Cref{sec:background} for a review of all technical terminology used in this introduction).
Although Regina~\cite{Regina} offers an \textsc{Unknot Recognition} algorithm that experimentally exhibits polynomial-time behaviour~\cite{burton2012unknot},
in theory this algorithm is exponential-time in the worst case.
There are other theoretical upper bounds on the computational complexity of \textsc{Unknot Recognition}, but these often use very different (and in some cases, much more complicated) techniques:
it is in both \NP~\cite{hass1999computational} and \coNP~\cite{kuperberg2014knottedness,lackenby2016efficient},
and Lackenby has announced a quasipolynomial-time algorithm \cite{LackenbyTalk}.
It remains unknown whether there is a (deterministic) polynomial-time algorithm for \textsc{Unknot Recognition}.

We are interested in algorithmically recognising other types of knots as well. Baldwin and Sivek give some complexity results for such other recognition problems in~\cite{baldwin2019complexity}:
they show that a range of natural problems are in \NP\ and/or \coNP\ (some unconditionally, others conditional on the Generalised Riemann Hypothesis).
Prior to the work in the present paper, none of these problems had algorithms with practical software implementations.

In this paper, we focus on the problem of deciding whether a knot is prime or composite, and in the composite case on computing its prime factorisation.
This problem arises from the classical result of Schubert~\cite{schubert49unique} that every knot has a unique prime factorisation. We study variants of this problem, in decreasing order of specificity:

\begin{problem}[\textsc{Knot Factorisation}]\label{problem:factorisation}
\hfill
\begin{description}
\item[\textup{\texttt{INPUT:}}]
A knot $K$.
\item[\textup{\texttt{OUTPUT:}}]
A collection of knots giving the prime factorisation of $K$.
\end{description}
\end{problem}

\begin{problem}[$k$-\textsc{Summands}]\label{problem:summands}
\hfill
\begin{description}
\item[\textup{\texttt{INPUT:}}]
A knot $K$, and a positive integer $k$.
\item[\textup{\texttt{QUESTION:}}]
Does the prime factorisation of $K$ have at least $k$ nontrivial summands?
\end{description}
\end{problem}

\begin{problem}[\textsc{Composite Knot Recognition}]\label{problem:composite}
\hfill
\begin{description}
\item[\textup{\texttt{INPUT:}}]
A knot $K$.
\item[\textup{\texttt{QUESTION:}}]
Is $K$ a composite knot?
\end{description}
\end{problem}

Most commonly, the input knot $K$ is presented as a knot diagram.
However, a diagram is difficult to work with in many settings, and the first step of many algorithms is to convert the diagram into a more useful form.
Often, the conversion is to a triangulation of the \textbf{knot exterior}:
the $3$-manifold obtained by deleting a small open neighbourhood of $K$ from the ambient $3$-sphere.
This is fine for decision problems like \Cref{problem:summands,problem:composite};
however, for \textsc{Knot Factorisation}, if we require the output to be in the form of knot diagrams -- rather than triangulations of knot exteriors -- then we have
the additional challenge of converting these triangulations to diagrams, which is itself a highly nontrivial problem~\cite{dunfield2022diagram}.

Our approach is to encode knots as \textbf{edge-ideal triangulations} (\textbf{H-triangulations} in \cite{aribi2023Htriangulation}).
For a knot $K$, an edge-ideal triangulation is a pair $(\tri,\ell)$, where $\tri$ is a triangulation of the $3$-sphere $\mathbb{S}^3$,
and $\ell$ is an embedded loop (called an \textbf{ideal loop}) of edges (called \textbf{ideal edges}) in $\tri$, such that $\ell$ is topologically equivalent to $K$.
This idea extends to triangulations $\tri$ of arbitrary compact $3$-manifolds, and collections of loops in the $1$-skeleton of $\tri$.

Embedding a knot in the $1$-skeleton of a triangulation is, on its own, not a new idea. For knots in $3$-manifolds this is a common choice of combinatorial encoding; e.g., see~\cite{agol2006computational,lackenby2016efficient}.
But even for knots in $\mathbb{S}^3$, this setting has been described and used in many publications with various backgrounds;
e.g., see~\cite{aribi2023Htriangulation,Benedetti12,BenedettiLutz:KnotsInCollapsible,burton2023counterexamples,ibarra2024Htriangulation,Lickorish1991} and the references therein.

The novelty of our approach is to view an ideal loop as representing a boundary component of a $3$-manifold,
loosely analogous to how an ideal triangulation (in the usual sense) has boundary components given by the ideal vertices.
More precisely, given an edge-ideal triangulation $(\tri,\ell)$ of a knot $K$, we consider a small open solid torus neighbourhood $N$ of $\ell$,
and we view $(\tri,\ell)$ as a combinatorial presentation of the $3$-manifold $\manifold$ given by deleting $N$ from $\tri$.
In this article, $\tri$ is a $3$-sphere and hence $\manifold$ is just the exterior of $K$.

To see why this perspective is useful, consider a normal surface $S$ in $\tri$.
The intersection of $S$ with the solid torus $N$ is a disjoint union of meridian discs, with one such disc for each point in $S\cap\ell$.
After deleting $N$, $S$ becomes a surface that intersects the torus boundary of $\manifold$ in curves that all have the same slope
(in our case, the slope of the meridian of the knot).
Hence, the edge-ideal triangulation carries more information than just the topology of $\manifold$: it also prescribes a slope along which normal surfaces must intersect the boundary of $\manifold$.

We use this setting to factorise knots in the $3$-sphere. However, we believe other applications for edge-ideal triangulations exist.
For instance, consider the problem of recognising Seifert fibred spaces (SFS) with nonempty boundary.
This was shown to be in \NP\ by Jackson~\cite{jackson2024seifert}, but the techniques are not well-suited to practical implementation.
SFS with boundary fit well into the framework of edge-ideal triangulations:
any such SFS $\manifold$ can be characterised using ``vertical'' essential annuli, which intersect components of $\partial \manifold$ in prescribed slopes.
Exploring this idea is work in progress.

The utility of edge-ideal triangulations goes beyond the observation about boundary slopes of surfaces.
It also allows us to apply the tool of \textbf{crushing} a normal surface, discussed in detail below. Crushing plays a central role in a number of practical $3$-manifold algorithms provided by Regina~\cite{burton2013regina,burton2014crushing,Regina}, such as unknot recognition or $3$-sphere recognition.

Our work extends this to give practical algorithms for \Cref{problem:factorisation,problem:summands,problem:composite}.
In \Cref{ssec:buildEdgeIdeal}, we give a straightforward linear-time procedure to convert
a diagram of a knot $K$ into an edge-ideal triangulation of $K$. Hence, our techniques
apply equally to diagrams, as well as edge-ideal triangulations of the exterior of the
input knot. However, converting an edge-ideal triangulation back into a knot diagram is
sufficiently difficult to be beyond the scope of this paper. Thus, we settle for the following:

\begin{theorem}\label{thm:summandsAlgo}
There is an algorithm  for \textup{\textsc{Knot Factorisation}} that
returns the prime factorisation of the input knot $K$ as a collection of edge-ideal triangulations.
\end{theorem}

The algorithm (see \Cref{algo:summands}) relies heavily on the tool of crushing normal surfaces.
For our implementation, we use Regina's~\cite{Regina} longstanding crushing functionality, which is already known to be effective in practice.
As it turns out, this practicability extends to our implementation: we are able to run our algorithm on prime knots from the census up to $19$ crossings~\cite{burton2020knots},
as well as composite knots constructed by summing up to $8$ knots sampled from the census (hence, knots with up to $152$ crossings), including some ``hard'' diagrams of composite knots with up $70$ crossings.
See \Cref{ssec:experiment} for details on these experiments.

In addition to the practical implications, we also adapt our techniques to obtain theoretical upper bounds on computational complexity. Specifically, we prove the following:

\begin{theorem}\label{thm:summandsNP}
\Cref{problem:summands} ($k$\textup{-\textsc{Summands}}) is in \NP.
\end{theorem}

Since \Cref{problem:composite} is a special case of \Cref{problem:summands}, this yields an alternative proof of one of Baldwin and Sivek's results~\cite[Theorem~7.3]{baldwin2019complexity}:
\textsc{Composite Knot Recognition} is in \NP.

The operation of crushing a normal surface is crucial to both the practicality of \Cref{algo:summands} and the proof of \Cref{thm:summandsNP}, so we now elaborate on its role in our work.

Inspired by unpublished work of Casson, crushing was first studied in detail by Jaco and Rubinstein~\cite{jaco2003efficient},
and later reformulated into a more implementation-friendly framework by Burton~\cite{burton2014crushing}.
To date, one limitation of crushing is that the theory remains relatively poorly developed for surfaces other than $2$-spheres and discs (though see~\cite{BHHP2024arXiv,Ichihara2023Crush}).
At first glance, this is an obstacle for the applications we have mentioned.
For a knot $K$, testing whether $K$ is composite entails finding a certain essential
annulus (in the exterior of $K$).
The same applies to recognising SFS with boundary.

By working with an edge-ideal triangulation $(\tri,\ell)$, we circumvent this issue:
the annuli show up in $\tri$ as $2$-spheres that intersect $\ell$ in two points.
Thus, instead of crushing annuli, we are able to work in the better-understood setting of crushing $2$-spheres.
In \Cref{sec:identifyingIdealEdges}, we develop the theory required to apply crushing in the presence of ideal loops.

The reason crushing is useful for practical computations is that it never increases the number of tetrahedra in a triangulation. This allows us to factor a knot efficiently: we can repeatedly crush surfaces to progressively ``decompose'' the knot, whilst simultaneously keeping tight control on the amount of data needed to encode the resulting factorisation.

Another feature of our work is that we work exclusively with \textbf{quad vertex} $2$-spheres in our edge-ideal triangulations;
we show that suitable $2$-spheres exist in \Cref{sec:existenceSpheres}.
Theoretically, this is necessary to control how crushing affects the ideal loop;
see \Cref{sec:identifyingIdealEdges}.
But even if this theoretical challenge were not present, working with quad (instead of standard) coordinates is often crucial for improving the practical performance of algorithms involving normal surfaces.

We combine the aforementioned ideas to formulate our main algorithm (\Cref{algo:summands}) in \Cref{sec:algo};
we prove correctness of the algorithm, discuss some implementation details, and summarise our experimental results.
In \Cref{sec:np}, we show how our techniques also give a proof of \Cref{thm:summandsNP}.
Finally, we discuss further applications in relation to triangulation complexity in \Cref{sec:further}.

\section{Background}\label{sec:background}

\subsection{Knots}\label{ssec:knots}

A \textbf{knot} is an embedding of a circle $K: S^1 \to \mathbb{R}^3$ into $3$-space.
We sometimes embed knots into the $3$-sphere $\mathbb{S}^3$, instead of $\mathbb{R}^3$;
since $\mathbb{S}^3$ is the one-point compactification of $\mathbb{R}^3$,
this has the benefit of making our topological space compact, but otherwise makes very little difference to the theory.
Two knots are \textbf{equivalent}, or of the \textbf{same type}, if there is an ambient isotopy $\Phi: \mathbb{R}^3 \times [0,1] \to \mathbb{R}^3$ taking the image of one to the image of the other;
intuitively, two knots are equivalent if one can be deformed into the other without passing the knot through itself.
If a knot is equivalent to the trivial unknotted circle, we call it the \textbf{unknot}, or the \textbf{trivial knot};
otherwise, the knot is \textbf{nontrivial}.

\begin{figure}[htbp]
\centering
	\begin{minipage}[t]{0.175\textwidth}
	\centering
		\includegraphics[scale=0.5]{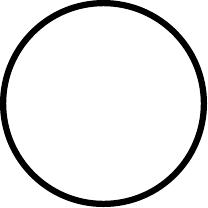}
	\subcaption{The unknot.}
	\label{subfig:unknot}
	\end{minipage}
	\hfill
	\begin{minipage}[t]{0.775\textwidth}
	\centering
		\begin{tikzpicture}
		\useasboundingbox (0,-1) rectangle (7.5,1);

		\node at (0,0) {
			\includegraphics[scale=0.5]{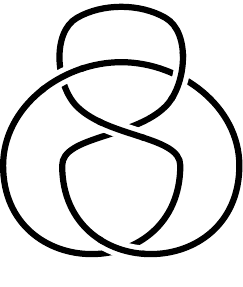}
		};
		\node at (2.5,0) {
			\includegraphics[scale=0.5]{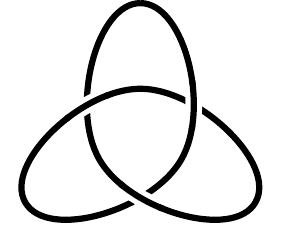}
		};
		\node at (6.5,0) {
			\includegraphics[scale=0.5]{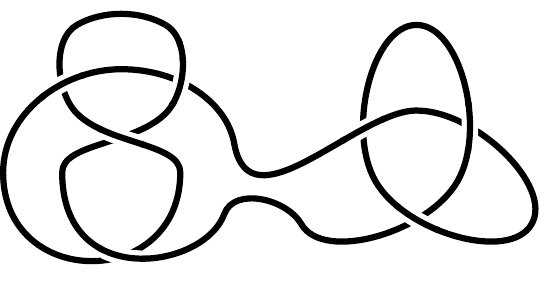}
		};
		\end{tikzpicture}
	\subcaption{Two nontrivial knots (left and middle), and a connected sum of them (right).}
	\label{subfig:connectedSum}
	\end{minipage}
\caption{Some examples of knot diagrams.}
\label{fig:knotDiagrams}
\end{figure}

We typically represent a knot by a \textbf{diagram}, obtained by generically projecting the knot into a plane,
while keeping information about which strand goes over and which goes under at \textbf{crossings} of the projection.
Given two knots $K_1$ and $K_2$, we can construct a \textbf{connected sum} $K= K_1 \# K_2$ by cutting both $K_1$ and $K_2$ and gluing the ends together to form a new knot.
A knot is \textbf{composite} if it is the connected sum of two nontrivial knots;
otherwise, the knot is \textbf{prime}.
Every nontrivial knot $K$ has a \textbf{factorisation} $K = K_1 \# \ldots \# K_\ell$ into
nontrivial prime knots $K_i$, $1\leqslant i \leqslant \ell$, and this factorisation is unique up to re-ordering~\cite{schubert49unique}.

\subsection{Triangulations of manifolds}\label{ssec:triangulations}

A \textbf{triangulation} $\tri$ (of \textbf{size} $|\tri|:=n$) is given by taking a finite set of disjoint tetrahedra $\Delta_1, \ldots , \Delta_n$, and gluing them
together along their $4n$ boundary triangles in pairs.
Let $\Phi_1 , \ldots , \Phi_{2n}$ denote the gluings. The quotient space
$\tri = \{\Delta_1, \ldots , \Delta_n\} / \{\Phi_1 , \ldots , \Phi_{2n}\}$
forms a closed $3$-manifold, in which case we call $\tri$ a \textbf{$3$-manifold triangulation}, if and only if: \emph{(a)} no edge is identified to itself in reverse;
and \emph{(b)} for each vertex $v$ of $\tri$, the \textbf{link} of $v$ -- i.e.,
the frontier of a small regular neighbourhood of $v$ -- forms a $2$-sphere.

As a result of the gluings, multiple lower-dimensional faces of  $\{\Delta_1, \ldots , \Delta_n\}$ become identified and we refer to the equivalence class of such faces as a single face of the triangulation $\tri$.
In fact, we often consider \textbf{one-vertex triangulations}, where all $4n$ vertices of $\{\Delta_1, \ldots , \Delta_n\}$ are identified to a single vertex. Two $3$-manifold triangulations are said to be {\bf combinatorially isomorphic} if one of them can be turned into the other by a relabelling of its faces.

Let $\manifold$ be a $3$-manifold with non-empty boundary $\partial \manifold$ a collection of tori. An {\bf ideal triangulation} of $\manifold$ is a triangulation with one vertex per component of $\partial \manifold$ such that its vertex link triangulates the boundary component, and $\manifold $ is homeomorphic to the underlying set of the triangulation with a neighbourhood of the vertices removed.

\subsection{Normal surface theory}\label{ssec:normalSurfaceTheory}

Consider a disc $D$ properly embedded in a tetrahedron $\Delta$ so that its boundary $\gamma$ is in \emph{general position} --
i.e., $\gamma$ is disjoint from the vertices of $\Delta$, and intersects the edges of $\Delta$ transversely.
The \textbf{weight} of $D$ in $\Delta$ is the number of points in which $D$ intersects the edges of $\Delta$.

A \textbf{normal surface} $S$ in a triangulation $\tri$ is a properly embedded surface such that:
\begin{itemize}
\item
$S$ is in general position:
it is disjoint from the vertices of $\tri$, and it intersects the edges and triangular faces of $\tri$ transversely.
\item
For each tetrahedron $\Delta$ of $\tri$,
the intersection $\Delta\cap S$ consists of a (possibly empty) disjoint union of finitely many positive-weight discs,
called \textbf{elementary discs}, such that the boundary of each elementary disc intersects each edge of $\Delta$ at most once.
\end{itemize}
The combinatorics of a normal surface remains unchanged under a \textbf{normal isotopy}:
an ambient isotopy that preserves each vertex, edge, face and tetrahedron of the triangulation.
Up to normal isotopy, each tetrahedron admits seven types of elementary discs:
four \textbf{triangle} types and three \textbf{quadrilateral} (or \textbf{quad}) types (see \Cref{fig:elemDiscs}).

\begin{figure}[htbp]
\centering
	\begin{minipage}[t]{0.32\textwidth}
	\centering
		\includegraphics[scale=0.5]{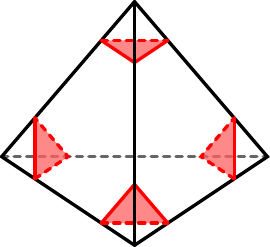}
	\subcaption{The four triangle types.}
	\label{subfig:elemTriangles}
	\end{minipage}
	\hfill
	\begin{minipage}[t]{0.67\textwidth}
	\centering
	\begin{tikzpicture}
		\node at (0,0) {
			\includegraphics[scale=0.5]{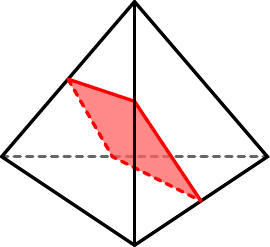}
		};
		\node at (2.5,0) {
			\includegraphics[scale=0.5]{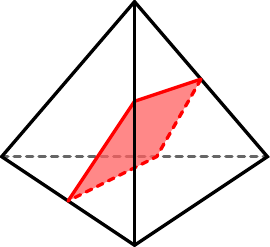}
		};
		\node at (5,0) {
			\includegraphics[scale=0.5]{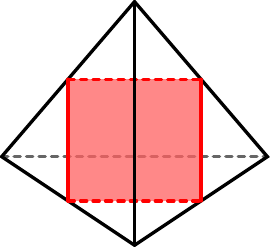}
		};
	\end{tikzpicture}
	\subcaption{The three quad types.}
	\label{subfig:elemQuads}
	\end{minipage}
\caption{The seven types of elementary disc.}
\label{fig:elemDiscs}
\end{figure}

The link of a vertex $v$ is an example of a normal surface. It is built entirely from elementary triangles (see \Cref{fig:vertexLink});
conversely a normal surface consisting only of triangles must be a union of vertex-linking components. With this in mind, we call a normal surface \textbf{trivial} if it only has triangles, and \textbf{nontrivial} if it has at least one quad.

\begin{figure}[htbp]
\centering
	\includegraphics[scale=1]{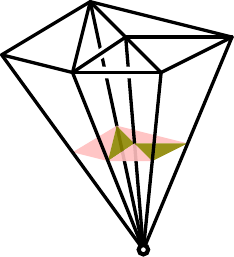}
\caption{The link of a vertex is a normal surface built entirely from triangles.}
\label{fig:vertexLink}
\end{figure}

Let $\tri$ be an $n$-tetrahedron triangulation.
Fix an arbitrary ordering of the tetrahedra, and in each tetrahedron fix an arbitrary ordering of the elementary triangle and quad types.
Consider a normal surface $S$ in $\tri$.
For $i\in\{1,\ldots,n\}$, $j\in\{1,2,3,4\}$ and $k\in\{1,2,3\}$, let $\Delta_i$ denote the $i$th tetrahedron,
let $t_{i,j}$ denote the number of triangles of the $j$th type in $\Delta_i$,
and let $q_{i,k}$ denote the number of quads of the $k$th type in $\Delta_i$.
The \textbf{standard coordinates} of $S$ are given by the $7n$-dimensional vector
\[
\mathbf{v}(S) :=
( t_{1,1},\, t_{1,2},\, t_{1,3},\, t_{1,4},\,
q_{1,1},\, q_{1,2},\, q_{1,3};\,
\ldots;\,
t_{n,1},\, t_{n,2},\, t_{n,3},\, t_{n,4},\,
q_{n,1},\, q_{n,2},\, q_{n,3} ).
\]
The \textbf{quad coordinates} of $S$ are given by the $3n$-dimensional vector
\[
\mathbf{q}(S) :=
( q_{1,1},\, q_{1,2},\, q_{1,3};\,
\ldots;\,
q_{n,1},\, q_{n,2},\, q_{n,3} ).
\]

Two normal surfaces are equivalent up to normal isotopy if and only if their standard coordinate vectors are equal.
However, this is not quite true for quad coordinates, since the corresponding vectors do not change if we add or remove vertex-linking components.
Call a normal surface \textbf{canonical} if it does not contain any vertex-linking components.
Two \emph{canonical} normal surfaces are equivalent up to normal isotopy if and only if their quad coordinate vectors are equal.

Not every vector in $\mathbb{Z}^{7n}$ describes the standard coordinates of a normal surface.
A vector
\[
\mathbf{v} =
( t_{1,1},\, t_{1,2},\, t_{1,3},\, t_{1,4},\,
q_{1,1},\, q_{1,2},\, q_{1,3};\,
\ldots;\,
t_{n,1},\, t_{n,2},\, t_{n,3},\, t_{n,4},\,
q_{n,1},\, q_{n,2},\, q_{n,3} )
\in \mathbb{Z}^{7n}
\]
is called \textbf{admissible} if it satisfies the following three necessary conditions:
\begin{itemize}
\item
Each component of $\mathbf{v}$ is non-negative.
\item
$M_\mathrm{S}\mathbf{v} = \mathbf{0}$, where $M_\mathrm{S}$ is the matrix of \textbf{standard matching equations} for $\tri$.
Roughly, this condition forces the elementary discs to ``match up'' across triangular faces of $\tri$,
and hence ensures that the union of all the elementary discs gives a well-defined surface. See \cite{hass1999computational} for more details.
\item
$\mathbf{v}$ satisfies the \textbf{quad constraints}:
for each $i\in\{1,\ldots,n\}$, at most one of the components $q_{i,1}$, $q_{i,2}$ or $q_{i,3}$ is nonzero.
This condition ensures that we can arrange the elementary discs in each tetrahedron so that
no two such discs intersect, and hence ensures that the resulting normal surface is embedded.
\end{itemize}
In fact, it turns out that $\mathbf{v}$ is a standard coordinate vector if and only if it is admissible.

Similarly, it turns out that
$
\mathbf{q} =
( q_{1,1},\, q_{1,2},\, q_{1,3};\,
\ldots;\,
q_{n,1},\, q_{n,2},\, q_{n,3} )
\in \mathbb{Z}^{3n}
$
is a quad coordinate vector if and only if it is \textbf{admissible}, meaning that:
\begin{itemize}
\item
Each component of $\mathbf{q}$ is non-negative.
\item
$M_\mathrm{Q}\mathbf{q} = \mathbf{0}$, where $M_\mathrm{Q}$ is the matrix of \textbf{quad matching equations} for $\tri$.
Roughly, this condition ensures that the quads around each edge of $\tri$ are compatible.
See \cite{Burton09Converting} for more details on quad matching equations.
\item
$\mathbf{q}$ satisfies the \textbf{quad constraints}.
\end{itemize}

\subsection{Vertex normal surfaces and exchange annuli}
\label{ssec:vertexAndExchange}

Although $\tri$ admits infinitely many (possibly disconnected) normal surfaces, it is often sufficient to restrict our attention to
the ``vertex normal surfaces'', which form a finite and tractable ``basis'' for the set of all normal surfaces.
To make this precise, notice that the first two requirements for admissibility (in either standard or quad coordinates) are homogeneous linear constraints,
so they define a polyhedral cone (in $\mathbb{R}^{7n}$ or $\mathbb{R}^{3n}$).
By scalar multiplication, we can project this cone down to a convex polytope consisting only of vectors whose coordinates sum to $1$;
we denote this polytope by $S(\tri)$ in standard coordinates, and by $Q(\tri)$ in quad coordinates.
A normal surface $S$ is \textbf{(standard) vertex} if $\mathbf{v}(S)$ is a scalar multiple of a vertex of $\mathrm{S}(\tri)$,
and no other vector in $\mathbb{Z}^{7n}$ is a smaller positive multiple of the same vertex.\footnote{%
Jaco and Tollefson~\cite{jaco1995algorithms} use slightly different terminology:
they refer to standard vertex normal surfaces as ``vertex solutions'',
and they use the term ``vertex surface'' to mean a connected \emph{two-sided} surface that is a multiple of a vertex of $\mathrm{S}(\tri)$.}
Similarly, $S$ is \textbf{quad vertex} if $\mathbf{q}(S)$ is a scalar multiple of a vertex of $\mathrm{Q}(\tri)$,
and no other vector in $\mathbb{Z}^{3n}$ is a smaller positive multiple of the same vertex.

For algorithmic purposes, one important feature of vertex normal surfaces is that the number of bits needed to encode their coordinates is polynomial in $|\tri|$.
Hass, Lagarias and Pippenger originally proved this in standard coordinates~\cite[Lemma~6.1]{hass1999computational},
but the same also holds in quad coordinates as an immediate consequence of the fact that every quad vertex surface is also standard vertex~\cite[Lemma 4.5]{Burton09Converting}.

Rather than working directly with the definition of vertex normal surfaces, it is often convenient to use other equivalent characterisations.
In particular, we use characterisations involving sums of surfaces and exchange annuli, so we now review these concepts.

The sum of two normal surfaces $S$ and $S'$ can be viewed through a purely linear algebraic lens:
provided $S\cup S'$ satisfies the quad constraints, the vector $\mathbf{v}(S)+\mathbf{v}(S')$
is admissible, and hence gives coordinates for a corresponding surface $S+S'$.
For a topological perspective, consider the components $\alpha_1, \ldots, \alpha_k$ of $S\cap S'$.
Assuming again that $S\cup S'$ satisfies the quad constraints, it turns out that we can obtain the same normal surface $S+S'$ by
performing cut-and-paste operations known as \textbf{regular exchanges} along each of the curves $\alpha_i$.
As illustrated in \Cref{fig:regExchange}, the ``location'' of a regular exchange can be tracked by augmenting
$S+S'$ with an ``exchange surface'' given by a union of $0$-weight discs.
See~\cite[Section~2]{jaco1995algorithms} for a more precise and detailed description of these ideas.
In our setting, the exchange surfaces are always annuli.

\begin{figure}[htbp]
\centering
	\begin{minipage}[t]{0.45\textwidth}
	\centering
		\includegraphics[scale=1]{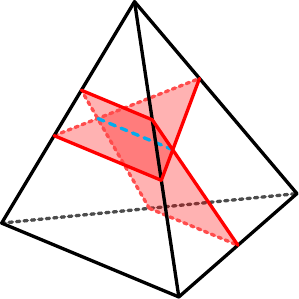}
	\subcaption{Before:
	The elementary discs intersect along an arc (blue).}
	\label{subfig:regExchangeArc}
	\end{minipage}
	\hfill
	\begin{minipage}[t]{0.45\textwidth}
	\centering
		\includegraphics[scale=1]{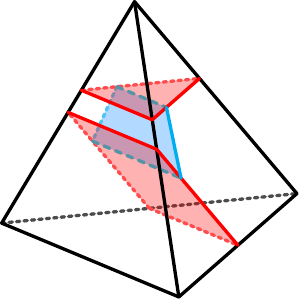}
	\subcaption{After:
	The ``location'' of the exchange is tracked by a $0$-weight disc (blue).}
	\label{subfig:exchangeSurfaceZeroWeightDisc}
	\end{minipage}
\caption{Example of a regular exchange inside a single tetrahedron.}
\label{fig:regExchange}
\end{figure}

More generally, we can speak of exchange annuli without any reference to a sum of surfaces.
In detail, let $S$ be a normal surface in a $3$-manifold triangulation $\tri$.
An annulus $A$ embedded in $\tri$ is an \textbf{exchange annulus} for $S$ if it satisfies the following conditions:
\begin{itemize}
\item
$A\cap S = \partial A$.
\item
$A$ has an orientable regular neighbourhood $N(A)$ in $\tri$.
\item
For every tetrahedron $\Delta$ of $\tri$, each component of $\Delta\cap A$ is a $0$-weight disc $D$ spanning two distinct elementary discs $E_1$ and $E_2$ in $\Delta$, such that:
	\begin{itemize}
	\item
	$\partial D = D\cap (\partial\Delta\cup E_1\cup E_2)$; and
	\item
	for each $i\in\{1,2\}$, $D\cap E_i$ is an arc whose endpoints lie in the interiors of two distinct triangular faces of $\Delta$.
	\end{itemize}
\end{itemize}

With all this in mind, we can finally state some different characterisations of vertex normal surfaces.
A normal surface $S$ is standard vertex if and only if multiples of $S$ are the only choices of surfaces $X$ and $Y$ that satisfy an equation of the form $kS = X+Y$.
Similarly, $S$ is quad vertex if and only if multiples of $S$ are the only choices of canonical surfaces $X$ and $Y$ that satisfy
an equation of the form $kS+L = X+Y$, where $L$ is a (possibly empty) disjoint union of vertex links.

Jaco and Tollefson~\cite{jaco1995algorithms} showed that, at least for normal $2$-spheres,
the characterisation in terms of sums can be used to derive a characterisation involving exchange annuli.
To define the necessary terminology, let $A$ be an exchange annulus for a surface $S$.
A \textbf{patch} relative to $A$ is a connected subsurface $P$ of $S$ such that $\partial P = P\cap\partial A$.
In other words, a patch minus its boundary is one of the components of $S-(A\cap S)$.
A \textbf{normal isotopy of a patch} $P$ is a sequence of compatible normal isotopies of the elementary discs intersecting $P$.
Two patches $P$ and $P'$ relative to $A$ are \textbf{normally isotopic along} $A$ if there is
a normal isotopy of $P$ that carries $A$ to itself and carries $P$ to $P'$.

Consider an exchange annulus $A$ for a normal $2$-sphere $S$.
The two components of $\partial A$ bound two disjoint discs in $S$.
If these two discs are normally isotopic along $A$, then we say that $A$ is a \textbf{trivial} exchange annulus for $S$. We have the following statement.

\begin{theorem}[{\cite[Theorem~4.1]{jaco1995algorithms}}]\label{thm:jacoTollefsonExchange}
A normal $2$-sphere $S$ is a standard vertex normal surface if and only if every exchange annulus for $S$ is trivial.
\end{theorem}

\subsection{Crushing along normal surfaces}\label{ssec:intermediateCell}

In this section we go over the procedure of {\em crushing} a normal surface inside a triangulation of a $3$-manifold.
This procedure, which is instrumental for $3$-sphere recognition, as well as many other algorithms
in low-dimensional topology, was first described by Jaco and Rubinstein in~\cite{jaco2003efficient}, following earlier unpublished ideas of Casson.
The method has since been simplified, see~\cite{burton2014crushing}, and implemented in {\em Regina}~\cite{Regina} by Burton.

We begin with some terminology (partly based on terminology from~{\cite[p.~91]{jaco2003efficient}}) that will be helpful for describing both the combinatorial and topological effects of crushing:

\begin{definition}\label{defs:inducedCells}
Let $S$ be a normal surface in a triangulation $\tri$.
The surface $S$ divides each tetrahedron $\Delta$ of $\tri$ into a collection of \textbf{induced cells} of the following types:
\begin{itemize}
\item \textbf{Parallel cells} of two types (see Figure~\ref{fig:parallelCells}):
	\begin{itemize}
	\item \textbf{Parallel triangular cells}:
	These lie between two parallel triangles of $S$.
	\item \textbf{Parallel quad cells}:
	These lie between two parallel quads of $S$.
	\end{itemize}
\item \textbf{Non-parallel cells} of nine types:
	\begin{itemize}
	\item \textbf{Corner cells}: These are tetrahedra that lie between a single triangle of $S$ and a single vertex of $\Delta$.
	\item \textbf{Wedge cells} of three types (see Figure~\ref{fig:wedgeCells}):
	These only occur when $S$ meets $\Delta$ in one or more quads.
	In this case, if we ignore any parallel and corner cells in $\Delta$, the two cells left over are the wedge cells.
	\item \textbf{Central cells} of five types (see Figure~\ref{fig:centralCells}):
	These only occur when $S$ does not meet $\Delta$ in any quads.
	In this case, if we ignore any parallel and corner cells in $\Delta$, the single cell left over is the central cell.
	\end{itemize}
\end{itemize}
Moreover, the faces of these induced cells come in the following types:
\begin{itemize}
\item \textbf{Normal faces}:
These are triangular and quadrilateral faces that come from the normal surface $S$.
\item \textbf{Parallel faces}:
These are quadrilateral faces that lie in the triangles of $\tri$ between two parallel normal arcs of $S$.
Note that parallel faces only appear in parallel and wedge cells (see \Cref{fig:parallelCells,fig:wedgeCells}).
\item \textbf{Corner faces}:
These are triangular faces that lie between a single normal arc of $S$ and a single vertex of $\tri$.
\item \textbf{Central faces}:
There is exactly one central face corresponding to each triangle $f$ of $\tri$, and it is
precisely the face that remains if we ignore all the parallel and corner faces in $f$.
A central face has three to six sides, depending on the number of normal arc types in $f$ (see \Cref{fig:wedgeCells,fig:centralCells}).
\end{itemize}
Let $\manifold$ denote the underlying $3$-manifold of $\tri$,
and let $\manifold^\dagger$ denote the $3$-manifold obtained from $\manifold$ by cutting along $S$.
The induced cells naturally yield a cell decomposition $\mathcal{D}$ of $\manifold$.
Moreover, ungluing the normal faces of $\mathcal{D}$ yields a cell decomposition $\mathcal{D}^\dagger$ of $\manifold^\dagger$.
We say that the cell decompositions $\mathcal{D}$ and $\mathcal{D}^\dagger$, and any cell decompositions given by
components of $\mathcal{D}$ and $\mathcal{D}^\dagger$, are \textbf{induced} by the normal surface $S$.
\end{definition}

\begin{figure}[htbp]
\centering
	\begin{minipage}[t]{0.49\textwidth}
	\centering
		\includegraphics[scale=0.7]{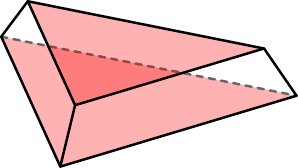}
	\subcaption{A parallel triangular cell.}
	\label{subfig:parallelTri}
	\end{minipage}
	\hfill
	\begin{minipage}[t]{0.49\textwidth}
	\centering
		\includegraphics[scale=0.7]{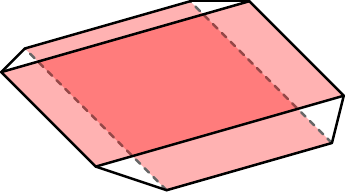}
	\subcaption{A parallel quad cell.}
	\label{subfig:parallelQuad}
	\end{minipage}
	\hfill
\caption{A normal surface $S$ can induce two types of parallel cells. The normal faces are shaded red.}
\label{fig:parallelCells}
\end{figure}

\begin{figure}[htbp]
\centering
	\begin{minipage}[t]{0.3\textwidth}
	\centering
		\includegraphics[scale=0.7]{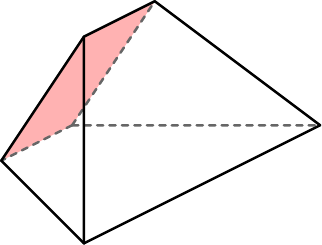}
	\subcaption{A wedge cell with no parallel faces.}
	\label{subfig:wedge0}
	\end{minipage}
	\hfill
	\begin{minipage}[t]{0.3\textwidth}
	\centering
		\includegraphics[scale=0.7]{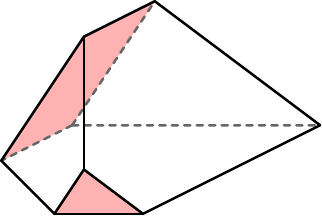}
	\subcaption{A wedge cell with one parallel face.}
	\label{subfig:wedge1}
	\end{minipage}
	\hfill
	\begin{minipage}[t]{0.3\textwidth}
	\centering
		\includegraphics[scale=0.7]{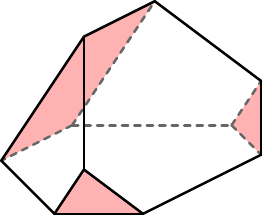}
	\subcaption{A wedge cell with two parallel faces.}
	\label{subfig:wedge2}
	\end{minipage}
\caption{A normal surface $S$ can induce three types of wedge cells. The normal faces are shaded red.}
\label{fig:wedgeCells}
\end{figure}

\begin{figure}[htbp]
\centering
	\begin{tikzpicture}

	\node at (0,0) {
		\includegraphics[scale=0.55]{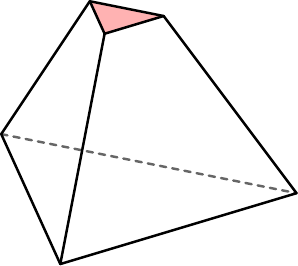}
	};
	\node at (3,0) {
		\includegraphics[scale=0.55]{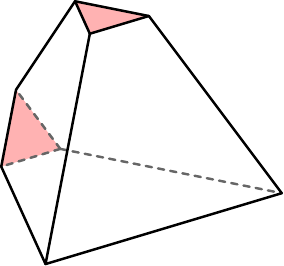}
	};
	\node at (6,0) {
		\includegraphics[scale=0.55]{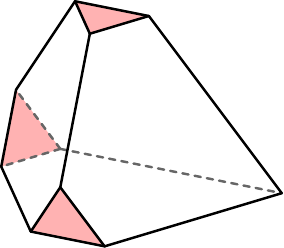}
	};
	\node at (9,0) {
		\includegraphics[scale=0.55]{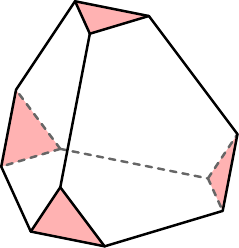}
	};

	\end{tikzpicture}
\caption{A normal surface $S$ can induce five types of central cell:
either a tetrahedron, or one of the four non-tetrahedron cells shown here.
The normal faces are shaded red.}
\label{fig:centralCells}
\end{figure}

With this terminology in mind, we can now summarise the procedure for crushing a normal surface $S$.
We first cut along $S$ and then shrink each \textbf{remnant} of $S$ (i.e., each component of the surface given by the union of all the normal faces after cutting) to a point.
This yields a cell decomposition whose cells are obtained from the induced cells by shrinking each normal face to a point.
Then, to complete the crushing procedure, we collapse non-tetrahedron cells until we recover a (possibly empty, possibly disconnected) triangulation $\tri'$.
We will discuss how crushing changes the topology of $\tri$ shortly, but first we present a precise definition of crushing:

\begin{definition}[The crushing procedure~{\cite[Definition~1]{burton2014crushing}}]\label{defs:stages}
Let $S$ be a normal surface in a triangulation $\tri$.
Each of the following operations builds on the previous one:
\begin{enumerate}
    \item Cut along $S$, and let $\mathcal{D}$ denote the resulting induced cell decomposition.
    \item Using the quotient topology, collapse each remnant of $S$ to a point.
    This turns $\mathcal{D}$ into a new cell decomposition $\mathcal{D}'$
    with $3$-cells of the following four possible types
    (see Figure~\ref{fig:destructibleCellFlatten}):
    \begin{itemize}
        \item \textbf{3-sided footballs}, which are obtained from
        corner cells and parallel triangular cells;
        \item \textbf{4-sided footballs}, which are obtained from parallel quad cells;
        \item \textbf{triangular purses}, which are obtained from wedge cells; and
        \item \textbf{tetrahedra}, which are obtained from central cells.
    \end{itemize}
    We say that $\mathcal{D}'$ is obtained by \textbf{non-destructively crushing}~$S$.
    Also, if a cell decomposition $\mathcal{D}^\ast$ is built entirely from $3$-cells of the four types listed above
    (even if it was not directly obtained by non-destructive crushing),
    then we call $\mathcal{D}^\ast$ a \textbf{destructible} cell decomposition.
    \item\label{step:flatten} To recover a triangulation from a destructible cell decomposition $\mathcal{D}^\ast$,
    we first build an intermediate cell complex $\mathcal{C}^\ast$ by using the quotient topology to flatten:
    \begin{itemize}
        \item all 3-sided and 4-sided footballs to edges; and
        \item all triangular purses to triangular faces.
    \end{itemize}
    This is illustrated in Figure~\ref{fig:destructibleCellFlatten}.
    Since triangulations are defined only by face gluings between tetrahedra,
    we \textbf{extract} a triangulation $\tri^\ast$ from $\mathcal{C}^\ast$ by:
    \begin{itemize}
        \item deleting all \textbf{isolated} edges and triangles that do not belong to any tetrahedra; and
        \item separating pieces of the cell complex that are only joined together along
        \textbf{pinched} edges.
    \end{itemize}
    This is illustrated in Figure~\ref{fig:extractTri}.
    We say that $\tri^\ast$ is obtained by \textbf{flattening} $\mathcal{D}^\ast$.
    Consider the triangulation $\tri'$ obtained by flattening
    the cell decomposition $\mathcal{D}'$ that results from non-destructively crushing $S$;
    we say that $\tri'$ is obtained by \textbf{(destructively) crushing}~$S$.
\qedhere
\end{enumerate}
\end{definition}

\begin{figure}[htb]
\centering
	\begin{minipage}[t]{0.25\textwidth}
	\centering
		\begin{tikzpicture}

		\node[inner sep=0pt] (football) at (0,0) {
			\includegraphics[scale=0.45]{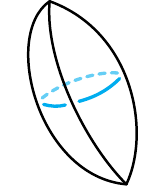}
		};
		\node[inner sep=0pt] (edge) at (2,0) {
			\includegraphics[scale=0.45]{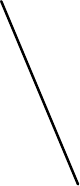}
		};
		\draw[thick, line cap=round, -Stealth] ($(football.east)+(-0.07,0)$) -- ($(edge.west)+(0.25,0)$);

		\end{tikzpicture}
	\subcaption{Flattening a $3$-sided football to an edge.}
	\label{subfig:football3}
	\end{minipage}
	\hfill
	\begin{minipage}[t]{0.25\textwidth}
	\centering
		\begin{tikzpicture}

		\node[inner sep=0pt] (football) at (0,0) {
			\includegraphics[scale=0.45]{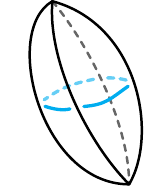}
		};
		\node[inner sep=0pt] (edge) at (2,0) {
			\includegraphics[scale=0.45]{CrushFigures/FlattenFootball.pdf}
		};
		\draw[thick, line cap=round, -Stealth] ($(football.east)+(-0.07,0)$) -- ($(edge.west)+(0.25,0)$);

		\end{tikzpicture}
	\subcaption{Flattening a $4$-sided football to an edge.}
	\label{subfig:football4}
	\end{minipage}
	\hfill
	\begin{minipage}[t]{0.45\textwidth}
	\centering
		\begin{tikzpicture}

		\node[inner sep=0pt] (purse) at (0,0) {
			\includegraphics[scale=0.4]{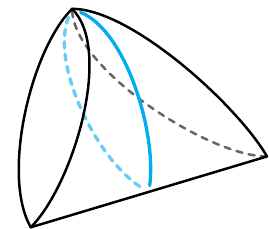}
		};
		\node[inner sep=0pt] (face) at (3.25,0) {
			\includegraphics[scale=0.4]{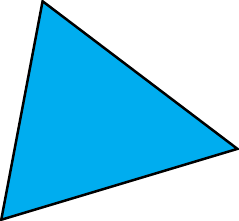}
		};
		\draw[thick, line cap=round, -Stealth] ($(purse.east)+(-0.12,0)$) -- ($(face.west)+(0.12,0)$);

		\end{tikzpicture}
	\subcaption{Flattening a triangular purse to a triangular face.}
	\label{subfig:triPurse}
	\end{minipage}
\caption{In addition to tetrahedra, a destructible cell decomposition $\mathcal{D}^\ast$ can contain three other types of $3$-cells.
To recover a triangulation from $\mathcal{D}^\ast$, we need to flatten the non-tetrahedron cells.}
\label{fig:destructibleCellFlatten}
\end{figure}

\begin{figure}[htb]
\centering
	\begin{tikzpicture}

	\node (before) at (0,0) {
		\includegraphics[scale=0.7]{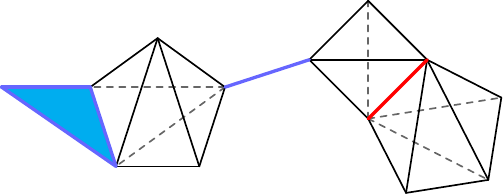}
	};
	\node (after) at (7,0) {
		\includegraphics[scale=0.7]{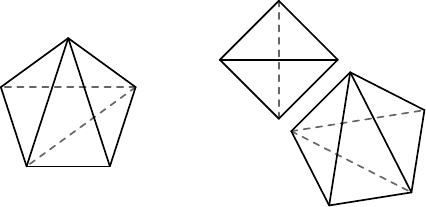}
	};
	\draw[thick, line cap=round, -Stealth] ($(before.east)+(0.05,0)$) -- ($(after.west)+(0.1,0)$);

	\end{tikzpicture}
\caption{Extracting a triangulation by deleting isolated edges and triangles (highlighted in blue), and separating pinched edges (highlighted in red).}
\label{fig:extractTri}
\end{figure}

An important feature of crushing is that whenever we crush a nontrivial normal surface in a triangulation $\tri$, we obtain a new triangulation $\tri^\ast$ with $|\tri^\ast| < |\tri|$.
This follows almost immediately from the definitions:
the tetrahedra in $\tri^\ast$ correspond to central cells in the induced cell decomposition,
and there is exactly one central cell for each tetrahedron of $\tri$ that contains no quads (and no central cell for all other tetrahedra).
Therefore, if we crush a surface with at least one quad, then any tetrahedron containing such a quad will not ``survive'' the crushing procedure.

Of course, this is not helpful unless we know how the topology of $\tri^\ast$ relates to the topology of $\tri$.
\emph{Non-destructively} crushing a normal $2$-sphere $S$ is straightforward to understand:
after cutting along $S$, shrinking the remnants is equivalent to capping off with $3$-balls.
In particular, if $S$ is a separating $2$-sphere, then we are simply decomposing the ambient $3$-manifold as a connected sum of two other $3$-manifolds.
However, understanding the topological effect of the flattening step requires a little more work.
The following lemma provides a helpful framework for studying the flattening step:

\begin{lemma}[Crushing lemma~{\cite[Lemma~3]{burton2014crushing}}]\label{lem:benCrushing}
Let $\tri^\ast$ be the triangulation given by flattening some destructible cell decomposition $\mathcal{D}^\ast$.
Then $\tri^\ast$ can be obtained from $\mathcal{D}^\ast$ by performing a sequence of
zero or more of the following \textbf{atomic moves} (see Figure~\ref{fig:atomicMoves}), one at a time, in some order:
\begin{itemize}
    \item flattening a \textbf{triangular pillow} to a triangular face;
    \item flattening a \textbf{bigon pillow} to a bigon face; and
    \item flattening a bigon face to an edge.
\end{itemize}
Since our cell decompositions are defined only by face gluings between $3$-cells,
after each atomic move we implicitly \textbf{extract} a cell decomposition by:
\begin{itemize}
    \item deleting all \textbf{isolated} vertices, edges, bigons and triangles that do not belong to any $3$-cells; and
    \item separating pieces of the cell complex that are only joined together along
    \textbf{pinched} edges.
\end{itemize}
\end{lemma}

\begin{figure}[htbp]
\centering
	\begin{minipage}[t]{0.4\textwidth}
	\centering
		\begin{tikzpicture}

		\node[inner sep=0pt] (pillow) at (0,0) {
			\includegraphics[scale=0.45]{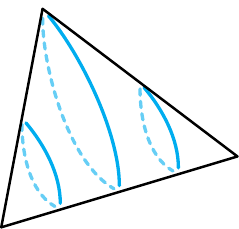}
		};
		\node[inner sep=0pt] (face) at (3.4,0) {
			\includegraphics[scale=0.45]{CrushFigures/FlattenTriPurse.pdf}
		};
		\draw[thick, line cap=round, -Stealth] ($(pillow.east)+(-0.32,0)$) -- ($(face.west)+(0.12,0)$);

		\end{tikzpicture}
	\subcaption{Flattening a triangular pillow.}
	\label{subfig:flattenTriPillow}
	\end{minipage}
	\hfill
	\begin{minipage}[t]{0.3\textwidth}
	\centering
		\begin{tikzpicture}

		\node[inner sep=0pt] (pillow) at (0,0) {
			\includegraphics[scale=0.45]{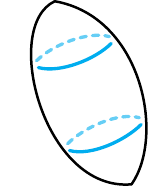}
		};
		\node[inner sep=0pt] (face) at (2.6,0) {
			\includegraphics[scale=0.45]{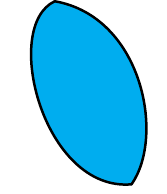}
		};
		\draw[thick, line cap=round, -Stealth] ($(pillow.east)+(-0.07,0)$) -- ($(face.west)+(0.25,0)$);

		\end{tikzpicture}
	\subcaption{Flattening a bigon pillow.}
	\label{subfig:flattenBigonPillow}
	\end{minipage}
	\hfill
	\begin{minipage}[t]{0.25\textwidth}
	\centering
		\begin{tikzpicture}

		\node[inner sep=0pt] (face) at (0,0) {
			\includegraphics[scale=0.45]{CrushFigures/FlattenBigon.pdf}
		};
		\node[inner sep=0pt] (edge) at (2,0) {
			\includegraphics[scale=0.45]{CrushFigures/FlattenFootball.pdf}
		};
		\draw[thick, line cap=round, -Stealth] ($(face.east)+(-0.07,0)$) -- ($(edge.west)+(0.25,0)$);

		\end{tikzpicture}
	\subcaption{Flattening a bigon face.}
	\label{subfig:flattenBigonFace}
	\end{minipage}
\caption{The three atomic moves for flattening a destructible cell decomposition.}
\label{fig:atomicMoves}
\end{figure}

As part of the proof of the crushing lemma, Burton showed in~\cite{burton2014crushing} that
if we are careful about the order in which we perform the atomic moves,
then we only ever encounter cell decompositions with $3$-cells of the following seven types:
3-sided footballs, 4-sided footballs, triangular purses, tetrahedra,
triangular pillows, bigon pillows, and \textbf{bigon pyramids} (see \Cref{fig:bigonPyramid}).

\begin{figure}[htbp]
\centering
	\includegraphics[scale=0.7]{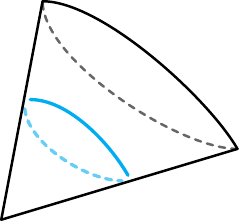}
\caption{A bigon pyramid.}
\label{fig:bigonPyramid}
\end{figure}

As suggested earlier, the crushing lemma often helps to overcome one of the key challenges in applications of crushing:
understanding how the flattening step changes the topology of the underlying $3$-manifold.
For instance, the crushing lemma was used in~\cite{burton2014crushing} to show that when the crushed surface is a $2$-sphere or disc, the changes in topology are mild, and can be easily detected.

\section{Existence of suitable quad vertex normal spheres}\label{sec:existenceSpheres}

In this section we prove that, in any edge-ideal triangulation $(\tri,\ell)$ of a composite knot, we can always find a normal $2$-sphere that we can use to either reduce the size of our input, or to witness a (possibly trivial) decomposition of the knot.
Crucially, we can ensure that these $2$-spheres are quad vertex, which allows us to:
\emph{(a)} enumerate candidate $2$-spheres efficiently, and \emph{(b)} control the effect of crushing such $2$-spheres.

Let $(\tri,\ell)$ be an edge-ideal triangulation of a knot $K$, as defined in \Cref{sec:intro}.
The ideal loop $\ell$ consists of a sequence of edges of $\tri$, which we denote by $\ell_1,\ldots,\ell_m$.
For any edge $e$ of $\tri$ and any normal surface $S$ in $\tri$, let $w_e(S)$ denote the
\textbf{edge-weight} of $S$ at $e$ -- i.e., the number of times $S$ intersects $e$.
Moreover, we define $w_\ell(S) = \sum_{i=1}^{m} w_{\ell_i}(S)$ -- i.e., the number of times $S$ intersects the loop $\ell$.
We have the following lemma:

\begin{restatable}{lemma}{lemNontrivialSphere}\label{lem:nontrivialS2}
Let $(\tri,\ell)$ be an edge-ideal triangulation of a knot $K$.
If $K$ is a composite knot, then $\tri$ contains a nontrivial normal $2$-sphere $S$ with $w_\ell(S)=2$.
\end{restatable}

\begin{proof}
Let $\manifold$ denote the knot exterior of $\ell$ in $\tri$, given by $\tri-N(\ell)$.
Since $\ell$ forms a composite knot, there exists a topologically embedded $2$-sphere $\tilde{S}$ in $\tri$ intersecting $\ell$ twice, and separating $\ell$ into two nontrivial summands; this $2$-sphere corresponds to an essential annulus $\tilde{A}$ in $\manifold$.
Normalising $\tilde{S}$ in $\tri$ yields a normal surface $\hat{S}$ via a sequence of the following operations:
\begin{itemize}
\item A compression along a $0$-weight compression disc;
\item An isotopy across an edge decreasing its edge-weight; and%
\item A deletion of a $0$-weight component.
\end{itemize}
Thus, each component of $\hat{S}$ is a $2$-sphere. Moreover, note that $w_
\ell(\hat{S})=2$, because the only way the normalisation procedure can change this weight is to isotope the surface across an edge $\ell_i$,
which is impossible since $\tilde{A}$ does not admit an essential $\partial$-compression disc.
Finally, since $\tri$ is a $3$-sphere and hence cannot contain any non-separating $2$-spheres, we see that $\hat{S}$ must include a ($2$-sphere) component $S$ with $w_\ell(S)=2$, as required.
\end{proof}

As mentioned at the beginning of this section, we strengthen \Cref{lem:nontrivialS2} to a statement about the existence of \emph{quad vertex} normal $2$-spheres;
see \Cref{cor:quadVertexComposite} below.
We organise the proof of this into two parts, which we present as \Cref{lem:x+y,lem:quadVertex}.

\begin{lemma}\label{lem:x+y}
Let $S$ be a normal $2$-sphere in a triangulation of $\mathbb{S}^3$.
If $S$ is not quad vertex, then one of the following holds, for some nontrivial normal surfaces $X$ and $Y$ that are not normally isotopic to $S$:
\begin{itemize}
\item
$2S = X+Y$, where $X$ and $Y$ are $2$-spheres.
\item
$S = X+Y$, where (without loss of generality) $X$ is a $2$-sphere and $Y$ is a torus.
\item
$S+L = X+Y$, where $L$ is a trivial $2$-sphere, and where $X$ and $Y$ are $2$-spheres.
\end{itemize}
\end{lemma}

\begin{proof}
If $S$ is not quad vertex, then we can write $nS+L = X+Y$, where:
\begin{itemize}
\item
$L$ is a (possibly empty) disjoint union of vertex links; and
\item
$X$ and $Y$ are (possibly disconnected) canonical surfaces such that neither $X$ nor $Y$ is a multiple of $S$.
\end{itemize}
We first show that we can choose $X$ and $Y$ so that they are both connected.
Adapting an argument of Jaco and Tollefson \cite{jaco1995algorithms}, we show that it suffices to choose $X$ and $Y$ so that:
\begin{itemize}
\item
the coefficient $n$ is minimised; and
\item
if there is more than one choice with minimum $n$, then the number $i(X,Y)$ of components in $X\cap Y$ is minimised.
\end{itemize}

Fix an arbitrary component $C$ of $X$, and write $X = C \sqcup X'$.
Suppose for the sake of contradiction that $X$ is disconnected, and hence that $X'$ is nonempty.
We have two cases:
\begin{itemize}
\item
If $C \cap Y = \emptyset$, then the fact that $C$ is canonical means that $C = S$.
But then we can write $(n-1)S + L = X' + Y$, contradicting minimality of $n$.
\item
If $C \cap Y \neq \emptyset$, then consider the canonical surface $Y'$ such that $C + Y = Y' + L'$, where $L'$ is a disjoint union of vertex links.
Observe that $L'$ must be a subset of $L$, so we have $nS + (L - L') = X' + Y'$.
Moreover, by construction, we have $i(X',Y') = i(X',Y) = i(X,Y) - i(C,Y) < i(X,Y)$, contradicting minimality of $i(X,Y)$.
\end{itemize}
Thus, $X$ must be connected.
By the same argument, $Y$ must also be connected.

Since $nS+L$ is a disjoint union of $2$-spheres, and since $X$ and $Y$ are connected, we have $2n \leqslant \chi(nS+L) = \chi(X+Y) \leqslant 4$, and hence $n \leqslant 2$.
Assuming without loss of generality that $\chi(X) \geqslant \chi(Y)$, and using the fact that $\tri$ is a $3$-sphere, we therefore have only three possible cases:
\begin{itemize}
\item
$2S=X+Y$, where $\chi(X)=\chi(Y)=2$;
\item
$S=X+Y$, where $\chi(X)=2$ and $\chi(Y)=0$; or
\item
$S+L=X+Y$, where $\chi(X)=\chi(Y)=2$ and $L$ is a single $2$-sphere.
\qedhere
\end{itemize}
\end{proof}

\begin{restatable}{lemma}{lemQuadVertex}\label{lem:quadVertex}
Let $(\tri,\ell)$ be an edge-ideal triangulation of a knot.
Moreover, let $\Omega$ denote the set of all nontrivial normal $2$-spheres $S$ in $\tri$ such that either $w_\ell(S)=0$ or $w_\ell(S)=2$.
If $\Omega$ is nonempty, then there must be a quad vertex normal surface in $\Omega$.
\end{restatable}

\begin{proof}
We start by defining a complexity on the normal surfaces of $\tri$ such that the
minimum complexity $2$-sphere in $\Omega$ must be the desired quad vertex normal surface. For this, let $F$ be a normal surface in $\tri$. The \textbf{complexity $c(F)$} of $F$ is the pair $\left( w_\ell(F), q(F) \right)$. For normal surfaces $F$ and $F'$ in $\tri$, we define the order $c(F) < c(F')$ if and only if either:
\begin{itemize}
\item $w_\ell(F) < w_\ell(F')$; or
\item $w_\ell(F) = w_\ell(F')$, and $q(F) < q(F')$ in the lexicographical order.
\end{itemize}
That is, we order $\left( w_\ell(F), q(F) \right)$ overall lexicographically.

Note that, by construction, if $F$ and $F'$ are both canonical, then $c(F) = c(F')$ if and only if $F$ is normally isotopic to $F'$.
Moreover, we have $c(F+F') = \left( w_\ell(F) + w_\ell(F'), q(F) + q(F') \right)$.

With this in mind, suppose we have $S\in\Omega$ that is not a quad vertex normal surface.
By \Cref{lem:x+y}, we only have three cases to consider, and in each case we find another normal $2$-sphere $X\in\Omega$ such that $c(X)<c(S)$:

\begin{description}
\item[Case 1: $2S=X+Y$, with $\chi(X)=\chi(Y)=2$.]\hfill\\
Assume without loss of generality that $c(X) \leqslant c(Y)$;
in particular, we have $w_\ell(X) \leqslant 2$, and hence $X \in \Omega$.
Since $X$ is not a multiple of $S$, we know in particular that $c(X) \neq c(S)$ which implies $c(X) < c(Y)$. Thus, additivity tells us that $c(X) < c(S)$.
\item[Case 2: $S=X+Y$, with $\chi(X)=2$ and $\chi(Y)=0$.]\hfill\\
Since $w_\ell(S) \leqslant 2$, it follows immediately that $w_\ell(X) \leqslant 2$, so we have $X \in \Omega$.
Now, additivity implies that $c(X) < c(S)$.
\item[Case 3: $S+L=X+Y$, with $\chi(X) = \chi(Y) = 2$.]\hfill\\
Assume without loss of generality that $c(X) \leqslant c(Y)$, and hence that $X \in \Omega$.
It remains to argue that $c(X) < c(S)$.
Since $q(S) = q(X) + q(Y)$, it is clear that $q(X) < q(S)$ in the lexicographical ordering.
Thus, all we need to show is that $w_\ell(X) \leqslant w_\ell(S)$:
	\begin{itemize}
	\item If $w_\ell(S) = 0$, then $w_\ell(X+Y) = w_\ell(S+L) = 2$.
	Since, by assumption, we have $w_\ell(X) \leqslant w_\ell(Y)$, the only possibility is that $w_\ell(X) = 0$.
	\item If $w_\ell(S) = 2$, then $w_\ell(X+Y) = w_\ell(S+L) = 4$.
	Since, by assumption, we have $w_\ell(X) \leqslant w_\ell(Y)$, we must have $w_\ell(X) \leqslant 2$.
	\end{itemize}
\end{description}

Altogether, if $S \in \Omega$ is not a quad vertex normal surface, then we have shown that we can find $X \in \Omega$ such that $c(X) < c(S)$.
Hence, if we choose $S \in \Omega$ minimising $c(S)$, then $S$ must be
a quad vertex normal surface.
\end{proof}

\begin{corollary}\label{cor:quadVertexComposite}
Let $(\tri,\ell)$ be an edge-ideal triangulation of a knot $K$.
If $K$ is composite, then $\tri$ contains a quad vertex $2$-sphere $S$ such that either $w_\ell(S)=0$ or $w_\ell(S)=2$.
\end{corollary}

\begin{proof}
To apply \Cref{lem:quadVertex}, we simply note that $\Omega$ is nonempty, by \Cref{lem:nontrivialS2}.
\end{proof}

\begin{corollary}
Let $(\tri,\ell)$ be an edge-ideal triangulation of a knot $K$.
If $K$ is composite, then $\tri$ contains a standard vertex $2$-sphere $S$ such that either $w_\ell(S)=0$ or $w_\ell(S)=2$.
\end{corollary}

\begin{proof}
This follows from~\cite[Lemma 4.5]{Burton09Converting}:
any quad vertex surface is also \emph{standard} vertex.
\end{proof}

\section{Tracking ideal edges under crushing}\label{sec:identifyingIdealEdges}

Consider an edge-ideal triangulation $(\tri,\ell)$ and a quad vertex $2$-sphere $S$ intersecting $\ell$ twice.
It is known that crushing $S$ only mildly changes the topology of $\tri$~\cite{burton2014crushing}.
In this section, we use a new characterisation of quad vertex $2$-spheres to show that crushing $S$ also only mildly changes the ideal loop $\ell$.
Moreover, using work of Agol, Hass, and Thurston~\cite[Sections 4 and 6]{agol2006computational}, we can track how $\ell$ changes in time polynomial in $|\tri|$;
this is purely theoretical, and in practice a much simpler and more naive tracking method is efficient enough.

\subsection{Orbits of segments}\label{ssec:orbitSegment}

Let $S$ be a normal surface in an edge-ideal triangulation $(\tri,\ell)$.
We say that $S$ splits an edge of $\tri$ into {\bf segments}.
We call such a segment \textbf{ideal} if it is a subset of the ideal loop $\ell$.
A \textbf{surviving} segment is a segment that is incident to a central cell.
This terminology is motivated by how crushing $S$ affects the segments:
each surviving segment becomes an edge of the resulting triangulation, whereas each non-surviving segment either gets destroyed or becomes identified with one of the surviving segments.
In other words, for each surviving segment $\sigma$, there is a natural ``orbit'' of segments that all become identified with $\sigma$ after crushing.
Thus, to algorithmically determine how crushing $S$ changes the ideal loop, we need to figure out which surviving segments have an orbit that contains an ideal segment.

It turns out to be useful to define the orbit of a segment $\sigma$ to be not just the collection of segments that get identified with $\sigma$ after crushing $S$,
but the entire subcomplex consisting of all the induced cells that get flattened and identified with $\sigma$ after crushing.
More precisely, we define the \textbf{induced orbit} $o(\sigma)$ of a segment $\sigma$ to be the smallest subcomplex of induced cells that contains:
\emph{(a)} the segment $\sigma$; \emph{(b)} every corner or parallel face incident to a segment in $o(\sigma)$; and \emph{(c)} every corner or parallel $3$-cell incident to a segment in $o(\sigma)$.
One reason for defining induced orbits in this way is that the topology of these subcomplexes captures information about the topological effect of crushing $S$;
another reason is that, in \Cref{ssec:trivialInducedOrbit,ssec:characteriseQuadVertexSphere},
we can use the topology of the induced orbits to give a new characterisation of quad vertex $2$-spheres.
An induced orbit containing an ideal segment is called an \textbf{ideal orbit}, and is of particular interest to us.

Before we discuss induced orbits further, we introduce some more useful terminology.
A segment is said to be \textbf{type-$k$} if $S$ intersects exactly $k$ endpoints of the segment. Thus, an edge that is disjoint from $S$ forms a type-$0$ segment, and an edge that intersects $S$ in $w$ points is split into two type-$1$ segments and $(w-1)$ type-$2$ segments.
Note that all the segments in an induced orbit must have the same type, so we may sensibly refer to the entire orbit as type-$k$, for some $k$.

With this terminology in mind, recall that our goal is to algorithmically determine which surviving segments belong to an ideal orbit.
The most obvious way to do this is to explicitly trace through the induced orbits;
in practice, this na\"{i}ve approach works well, and so this is what our implementation actually does.
However, in theory, an induced orbit can be very large:
in the worst case, the total number of segments in a type-$2$ orbit is proportional to the size of the normal coordinates;
even if we assume $S$ is a vertex normal surface, this could still be exponential in $|\tri|$~\cite[Lemma~6.1]{hass1999computational}.

To give a more satisfactory theoretical bound on the computational complexity of determining which surviving segments belong to each ideal orbit,
we turn to the weighted orbit-counting algorithm introduced by Agol, Hass and Thurston~\cite[Sections~4 and~6]{agol2006computational}.
This will allow us to accomplish our task in time polynomial in $|\tri|$.

In what follows, we first describe the Agol-Hass-Thurston weighted orbit-counting algorithm~\cite{agol2006computational}, together with its original application, before specifying how we apply this algorithm to our situation.

\paragraph*{The Agol-Hass-Thurston weighted orbit-counting algorithm}

Abstractly, the algorithm involves $k$ \textbf{pairings} (i.e., bijections that either preserve or reverse the ordering of the intervals) between sequences of consecutive integers in $[1,N] := \{ n\in\mathbb{Z} : 1\leqslant n\leqslant N \}$;
and a \textbf{weight function} $w:[1,N]\to\mathbb{Z}^d$ that partitions $[1,N]$ into $m$ subintervals of constant weights $z_1,\ldots,z_m \in \mathbb{Z}^d$.
For each orbit of $[1,N]$ under the pairings, the algorithm computes the sum of the weights of the orbit's elements.
Its running time is shown to be polynomial in $kmd \log{D} \log{N}$, where $D$ is the total weight (i.e., the sum of all weights combined).

\paragraph*{Computing the components of a normal surface}

In~\cite{agol2006computational}, the orbit-counting algorithm is applied to a normal surface $S$ with sum of coordinates equal to $W$ in a triangulation with $t$ tetrahedra.
The intersection points of $S$ with the edges of $T$ are labelled by the elements of $[1,N]$, with $N=W$.
Pairings are defined between intersection points that are connected by a normal arc;
see~\cite[Figure~4]{agol2006computational}.
Naturally, there are at most three such pairings per triangle, giving at most $3\cdot 4t = 12 t$ pairings overall.

The weight vector is of length $d=7t$, with one entry for each elementary disc type.
For each such disc type and a fixed edge intersecting it, $1$ is added to the appropriate coordinate in the weight vector of the appropriate intersection point.
For the number of subintervals with constant weight, note that every tetrahedron-edge incidence adds at most $6$ changes to the weight vector, and that there are at most $6t$ of those, hence $m \leqslant 36t$.
Finally, the total weight $D$ equals the sum of all normal coordinates $W$ of $S$.

With this setup, the algorithm computes orbit weights, which are equivalent to the normal coordinates of the connected components of $S$, in time polynomial in $t$ and $W$.

\paragraph*{Finding the surviving segments in ideal orbits}

To apply the orbit-counting algorithm to our setting, let $(\tri,\ell)$ be an edge-ideal triangulation with $t$ tetrahedra, and let $s$ be the number of edges in $\tri$;
note that $s \leqslant 6t$.
As before, let $S$ be a normal surface in $\tri$, and let $W$ denote the sum of its normal coordinates. The surface cuts the edges $e_1,\ldots,e_s$ of $\tri$ into segments
$\varepsilon_j$, $1 \leqslant j \leqslant \sum_{i=1}^{s} n_i =: N$, where the first $n_1$ segments are contained in $e_1$, the next $n_2$ in $e_2$, and so on.
Note that $N$ is bounded by a linear function in $t+W$.

We identify the set of segments with its index set $[1,N]$.
Our goal is to define pairings on $[1,N]$ such that two segments are in the same orbit under these pairings if and only if those segments are in the same induced orbit.
To this end, call two segments \textbf{adjacent} if they are contained in either a common corner face (this only happens for type-$1$ segments)
or a common parallel face (this only happens for type-$2$ segments).
Observe that two segments are in the same induced orbit if and only if they are transitively related via adjacency.
Thus, it suffices to construct a collection of pairings such that every pair of adjacent segments is witnessed by some pairing.

For this, first consider an arbitrary triangular face $F$ of $\tri$.
Within $F$, the corner and parallel faces -- each of which witnesses a pair of adjacent segments -- are in bijection with the normal arcs.
Thus, as illustrated in \Cref{fig:ahtTriangle}, we can capture all the adjacencies witnessed by $F$ using just three pairings -- one pairing for each normal arc type in $F$.

\begin{figure}[htbp]
\centering
	\includegraphics[scale=1]{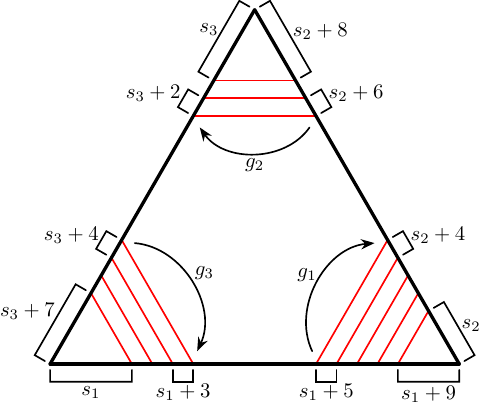}
\caption{%
In each triangular face, the normal arcs define three pairings.
For the example drawn here, all three pairings happen to be orientation-reversing:
$g_1:[s_1+5,s_1+9]\to[s_2,s_2+4]$,
$g_2:[s_2+6,s_2+8]\to[s_3,s_3+2]$, and
$g_3:[s_3+4,s_3+7]\to[s_1,s_1+3]$.
Adapted from Figure~4 of~\cite{agol2006computational}.}
\label{fig:ahtTriangle}
\end{figure}

Doing this for all of the triangular faces in $\tri$ gives a collection of pairings that witnesses all pairs of adjacent segments, as required.
Since $\tri$ has at most $4t$ faces, we have $k \leqslant 12t$ pairings in total.

For the weights, we assign a coordinate to each surviving or ideal segment, and essentially use the weight vectors as indicator functions.
In detail, label the ideal segments by $\rho_1,\ldots,\rho_{a}$, the surviving segments by $\sigma_1,\ldots,\sigma_{b}$, and set $d = a + b$.
We define a weight function $w: \{ 1, \ldots , N\} \to \mathbb{Z}^d$ for our segments, with the $j$th unit vector assigned to $\rho_j$, the $(a+i)$th unit vector assigned to $\sigma_i$, and the zero vector assigned to all other segments.

Recall that in our setting, $S$ intersects the ideal edges at most twice, so we can assume that $a \leqslant s + 2 \leqslant 6t + 2$.
We also assume that $S$ is nontrivial;
that is, crushing $S$ results in a smaller triangulation, and hence $b \leqslant 6(t-1)$.
Altogether, this yields $d < 12t$. In particular, we also have $D < 12t$ for the sum of all components of all weights.
It naturally follows that for this assignment of weights, the weight vector changes at most twice per non-zero assignment of weight vectors, and hence we have $m \leqslant 24t$.

After running the Agol-Hass-Thurston algorithm, we obtain one weight vector per orbit.
For each such weight vector, we can read off which ideal and surviving segments lie in the corresponding orbit:
nonzero entries in the first $a$ coordinates identify the ideal segments, and nonzero entries in the last $b$ coordinates identify the surviving segments.

Substituting the (crude) bounds we have on $k$, $m$, $d$, $D$ and $N$, we see that the running time is polynomial in $t\log(t+W)$. In particular, if $S$ is a (standard or quad) vertex normal surface, then the bounds from \cite{hass1999computational} imply bounds on $W$ such that this is a polynomial in $t$.

\subsection{Trivial induced orbits and cycles of faces}\label{ssec:trivialInducedOrbit}

Let $S$ be a normal $2$-sphere in a triangulation $\tri$. Consider the cell decomposition $\mathcal{D}$ obtained by non-destructively crushing $S$. We observe the following:
\begin{itemize}
\item Every $4$-sided football $F$ in $\mathcal{D}$ (see \Cref{subfig:football4}, left) comes from a parallel quad cell
(see \Cref{subfig:parallelQuad}).
Moreover, since we are crushing a connected surface, the two vertices of $F$ must be identified.

\item Every $3$-sided football $F$ in $\mathcal{D}$ (see \Cref{subfig:football3}, left) comes from either
a parallel triangular cell (see \Cref{subfig:parallelTri}) or a corner cell. Moreover, the two vertices of $F$ are identified if and only if $F$ came from a parallel triangular cell.

\item Every bigon face $F$ in $\mathcal{D}$ comes from either a parallel face or a corner face.
Moreover, the two vertices of $F$ are identified if and only if $F$ came from a parallel face.

\end{itemize}
Consider the subcomplex of $\mathcal{D}$ consisting only of the footballs and bigon faces.
From the above observations, we see that this subcomplex is equivalent to the subcomplex obtained from
applying non-destructive crushing to the union of the induced orbits of all the type-$1$ and type-$2$ segments.

Call a type-$1$ induced orbit \textbf{trivial} if said orbit forms a cone over a
homotopically trivial subcomplex of $S$ built from normal arcs and triangles.
In a similar vein, call a type-$2$ induced orbit \textbf{trivial} if said orbit forms a homotopically trivial product $I$-bundle.
The following observation is immediate:

\begin{lemma}\label{lem:flattenToEdge}
If $S$ is a nontrivial normal $2$-sphere such that every type-$1$ and type-$2$ segment has a trivial induced orbit, then
simultaneously flattening all of the footballs and bigon faces {\em (a)} leaves $\mathcal{D}$ topologically unchanged; and {\em (b)} flattens each induced orbit to a single edge.
\end{lemma}

The utility of \Cref{lem:flattenToEdge} comes from the following observations:
\begin{enumerate}
\item The conclusions of the lemma essentially say that the topological effect of crushing is easy to understand.
\item As established in the remainder of this section, the assumptions of the lemma are guaranteed to be satisfied by any quad vertex normal $2$-sphere.
\end{enumerate}

We first consider the case of type-$2$ segments.

\begin{lemma}\label{lem:cycleParallel}
Let $S$ be a standard or quad vertex normal $2$-sphere in some triangulation $\tri$, and let $P$ be an arbitrary cycle of parallel faces in the cell decomposition induced by $S$. Then $P$ forms an annulus (not a M\"{o}bius band), and the boundary curves $\partial P$ bound two disjoint discs in $S$ that are normally isotopic along $P$.
\end{lemma}

\begin{proof}
The proof of this statement is a straightforward adaptation of the proof of \cite[Theorem 4.1]{jaco1995algorithms} by Jaco and Tollefson.

We first show that if $P$ forms a M\"{o}bius band, then $S$ cannot be a vertex normal surface.
Let $A$ denote the annulus given by the frontier of a small regular neighbourhood of $P$ in the complement of $S$. Note that $A$ is an exchange annulus for $S$.
The curves given by $\partial A$ bound two disjoint discs $D_1$ and $D_2$ in $S$ given by $S - N(\partial P)$.
We have two possibilities:
\begin{itemize}
\item By the work of Jaco and Tollefson, see \cite{jaco1995algorithms} and \Cref{ssec:vertexAndExchange}, if $D_1$ and $D_2$ are not normally isotopic along $A$, then $S$ is not a vertex normal $2$-sphere.
\item On the other hand, if $D_1$ and $D_2$ are normally isotopic along $A$, then we see that $S$ is the double of an embedded one-sided projective plane.
Thus, we again conclude that $S$ is not a vertex normal surface.%
\footnote{Of course, $\tri$ is usually a $3$-sphere in this paper, so it never contains embedded projective planes.
But this argument clarifies that the result applies to other $3$-manifolds as well.}
\end{itemize}

Thus, we conclude that $P$ must form an annulus.
In particular, the curves given by $\partial P$ bound two disjoint disc subsets $D_1$ and $D_2$ of $S$.
Isotope $P$ slightly to obtain a zero-weight annulus $E$ such that $\partial E$ bounds two disjoint discs containing $D_1$ and $D_2$.
Observe that $E$ is an exchange annulus for $S$.
Thus, if $S$ is a vertex normal $2$-sphere, then it follows immediately from the work of Jaco and Tollefson, see \cite{jaco1995algorithms} and \Cref{ssec:vertexAndExchange}, that $D_1$ and $D_2$ must be normally isotopic along $P$.
\end{proof}

\begin{corollary}\label{cor:parallelRegion}
If $S$ is a standard or quad vertex normal $2$-sphere in some triangulation $\tri$, then every type-$2$ segment has a trivial induced orbit.
\end{corollary}

\begin{proof}
Consider an arbitrary type-$2$ segment $\sigma$.
The induced orbit $O_\sigma$ of $\sigma$ is an $I$-bundle built from parallel cells and parallel faces.
By \Cref{lem:cycleParallel}, any cycle of parallel faces in $O_\sigma$ must form an annulus that cuts off a $D^2\times I$ subset of $O_\sigma$. In particular, every such cycle must be contractible in $O_\sigma$.
Moreover, $O_\sigma$ cannot be a trivial $I$-bundle over the $2$-sphere, since this requires $S$ to be disconnected.
The only remaining possibility is that $O_\sigma$ is a homotopically trivial product $I$-bundle, and the statement follows from the definition of a trivial induced orbit.
\end{proof}

For type-$1$ segments, we take a similar approach, except this time we consider an arbitrary disc $C$ formed from a cycle of corner faces in the induced cell decomposition.

\begin{lemma}\label{lem:cycleCorner}
Let $S$ be a quad vertex normal $2$-sphere in some triangulation $\tri$, and let $C$ be a disc formed from a cycle of corner faces. Then $\partial C$ must bound a disc subset of $S$ consisting entirely of triangles.
\end{lemma}

\begin{proof}
For the sake of notation, arbitrarily assign the numbers $0$ and $1$ to the two sides of the disc $C$.
The curve $\partial C$ divides $S$ into two discs $D_0$ and $D_1$, labelled so that $D_i$ meets $C$ on side $i$.
Suppose that $D_0$ and $D_1$ each contain at least one quad.
We show that, under this assumption, $S$ cannot be a quad vertex surface.

To this end, note that the interior of the disc $C$ contains a single vertex of $\tri$.
Let $L$ denote the vertex linking sphere of this vertex.
The disc $C$ also divides $L$ into two discs $E_0$ and $E_1$, where again we label so that $E_i$ meets $C$ on side $i$.
Since both $D_0$ and $D_1$ contain quads, we can form two nontrivial normal $2$-spheres $N_i := D_i\cup E_{1-i}$, $i=0,1$. By construction, we have $S + L = N_0 + N_1$, and hence $S$ is not a quad vertex normal sphere.\footnote{%
This argument shows that $S$ does not even need to be a vertex surface for the same conclusion to hold;
it is enough for $S$ to be a quad \emph{fundamental} normal $2$-sphere.}
\end{proof}

\begin{corollary}\label{cor:cornerRegion}
If $S$ is a quad vertex normal $2$-sphere, then every type-$1$ segment has a trivial induced orbit.
\end{corollary}

\begin{proof}
Consider an arbitrary type-$1$ segment $\sigma$.
The induced orbit $O_\sigma$ of $\sigma$ is built from corner cells and corner faces.
It follows immediately that $O_\sigma$ is a cone over a subcomplex $K$ of $S$ built from
normal arcs and triangles, and hence it suffices to show that $K$ is homotopically trivial.
By \Cref{lem:cycleCorner}, we know that every cycle of arcs in $K$ bounds a disc subset of $K$.
Moreover, $K$ cannot be a $2$-sphere, since this requires $S$ to have a vertex-linking component.
Thus, the only possibility is that $K$ is homotopically trivial.
\end{proof}

\subsection{Induced orbit characterisation of quad vertex \texorpdfstring{$2$}{2}-spheres}\label{ssec:characteriseQuadVertexSphere}

Our crushing techniques rely on \Cref{cor:parallelRegion,cor:cornerRegion}, which show that quad vertex $2$-spheres have trivial induced orbits.
The converse is also true:

\begin{restatable}{lemma}{lemQuadVertexOrbits}\label{lem:quadVertexOrbits}
A normal $2$-sphere is quad vertex if and only if every type-$1$ and every type-$2$ segment has a trivial induced orbit.
\end{restatable}

\begin{proof}
We already proved one direction in \Cref{cor:cornerRegion,cor:parallelRegion}.
To prove the inverse, let $S$ be a normal $2$-sphere, and suppose $S$ is not quad vertex.
By \Cref{lem:x+y}, there are only three cases to consider.

In the first two cases ($X+Y$ is equal to either $S$ or $2S$), we see that $S$ is not even a standard vertex normal surface.
Thus, by the characterisation of Jaco and Tollefson (see~\cite[Theorem~4.1]{jaco1995algorithms}, or
\Cref{thm:jacoTollefsonExchange} in this paper), $S$ admits a nontrivial exchange annulus $A$.
By definition, $A$ is a $0$-weight surface, but we can isotope $A$ so that it becomes an annulus built entirely from parallel faces.
These parallel faces all form part of the same type-$2$ orbit, and this orbit must be nontrivial because $A$ is nontrivial.

In the third case, we have $S+L = X+Y$, where $L$ is a vertex link, and where $X$ and $Y$ are nontrivial normal $2$-spheres.
Consider the collection of exchange annuli corresponding to the curves of intersection between $X$ and $Y$.
Since $X$ and $Y$ are distinct, this collection necessarily includes at least one nontrivial exchange annulus $A$.
If both components of $\partial A$ lie in $S$, then by the same argument as above there must be a nontrivial type-$2$ orbit.

Otherwise, we have one component of $\partial A$ (call it $\gamma_S$) in $S$, and the other (call it $\gamma_L$) in $L$.
As before, isotope $A$ so that it is built entirely from parallel faces.
After this isotopy, let $D_0$ and $D_1$ denote the two disc subsets of $S$ given by cutting $S$ along $\gamma_S$.
Since $A$ is nontrivial, neither $D_0$ nor $D_1$ can be isotopic along $A$ to a disc subset of $L$.
In other words, $D_0$ and $D_1$ must each contain at least one quad.
For the other boundary curve of $A$, observe that $\gamma_L$ bounds a cycle of corner faces, and that attaching these corner faces to $A$ gives a disc $C$.
If we now remove the vertex link $L$, then this disc $C$ becomes a cycle of corner faces bounded by the curve $\gamma_S$.
We claim that $C$ belongs to a (type-$1$) induced orbit $o_C$ that is nontrivial.
Indeed, if $o_C$ were trivial, then it would contain a cone over either $D_0$ or $D_1$, which is impossible since we already observed that each of these discs includes a quad.
\end{proof}

\subsection{The effect of crushing on ideal loops}\label{ssec:crushIdealLoop}

Let $(\tri,\ell)$ be an edge-ideal triangulation of a (possibly trivial) knot $K$.
We now study the effect of crushing quad vertex $2$-spheres that intersect the ideal loop $\ell$ in either zero or two points.
We begin with the case where there are zero intersection points:

\begin{restatable}{lemma}{lemCrushZero}\label{lem:crushW=0}
Let $(\tri,\ell)$ be an edge-ideal triangulation of a (possibly trivial) knot $K$, and let $S$ be a quad vertex $2$-sphere with $w_\ell(S)=0$.

Crushing $S$ yields a new triangulation $\tri'$ consisting of either one or two $3$-sphere components. Moreover, all the edges in $\ell$ are left untouched, and together form an embedded loop in a component $\tri_0'$ of $\tri'$, and $(\tri_0',\ell)$ forms an edge-ideal triangulation of $K$.
\end{restatable}

\begin{proof}
Let $e$ be an edge in $\ell$.
Since $w_\ell(S)=0$, $S$ must be disjoint from $e$, and every induced cell incident to $e$ is either a central cell, or a wedge cell with no parallel faces.

Suppose that $e$ is only incident to wedge cells.
Each endpoint of $e$ forms the centre of a disc built from a cycle of corner faces of these wedge cells.
The boundary of such a disc $C$ forms a closed curve in $S$.
Thus, since $S$ is disjoint from $\ell$, we can take the union of $C$ with a disc subset of $S$
to obtain an embedded $2$-sphere $S'$ that intersects $\ell$ in exactly one point.
In other words, we have found a non-separating $2$-sphere $S'$ in the $3$-sphere, a contradiction.

Hence, every edge in $\ell$ must be incident to at least one central cell, and must therefore survive after crushing $S$.

It remains to determine the effect of crushing on the topology of $\tri$.
Non-destructively crushing $S$ yields a cell decomposition $D^\dagger$ representing a disjoint union of two $3$-spheres.
Let $D^\ast$ denote the cell decomposition obtained by simultaneously flattening all of the footballs and bigon faces in $D^\dagger$.
Since $S$ was quad vertex, $D^\ast$ is still a disjoint union of two $3$-spheres by \Cref{lem:flattenToEdge}.

The only remaining non-tetrahedron cells in $D^\ast$ are triangular pillows.
We proceed by flattening these pillows one at a time.
Each such flattening either has no topological effect, or deletes a $3$-sphere component.
However, since we already know that $\ell$ survives the crushing of $S$, the component containing $\ell$ cannot be deleted. Altogether, crushing $S$ yields a non-empty triangulation $\tri'$ with up to two components, one of which, denoted by $\tri_0$, contains $\ell$.

Since $S$ and $\ell$ were disjoint, the knot type of $\ell$ in $\tri_0$ is unchanged;
in other words, $(\tri_0',\ell)$ forms an edge-ideal triangulation of $K$.
\end{proof}

Suppose now that we have a quad vertex $2$-sphere $S$ that intersects $\ell$ in exactly two points.
Let $\ell_1, \ldots, \ell_m$ denote the edges of $\tri$ that together form the ideal loop $\ell$.
If both intersection points are on the same edge, then $S$ splits $\ell$ into \textbf{ideal arcs} of the following types:
\begin{itemize}
\item
one \textbf{short} ideal arc, meaning that the arc is formed from a single type-$2$ segment; and
\item
one \textbf{long} ideal arc, meaning that the arc is formed from two type-$1$ segments together with a (potentially empty) sequence of type-$0$ segments.
\end{itemize}
Otherwise, if the two intersection points are on two distinct edges, then $S$ splits $\ell$ into two long ideal arcs.
With this terminology in mind, we prove the following:

\begin{restatable}{lemma}{lemCrushTwo}\label{lem:crushW=2}
Let $(\tri,\ell)$ be an edge-ideal triangulation of a (possibly trivial) knot $K$, and let $S$ be a quad vertex $2$-sphere with $w_\ell(S)=2$,
so that $S$ cuts $\ell$ into two ideal arcs $\alpha_1$ and $\alpha_2$, corresponding respectively to two (possibly trivial) knots $K_1$ and $K_2$ such that $K = K_1 \# K_2$.

Crushing $S$ yields a new triangulation $\tri'$ consisting of either one or two $3$-sphere components.
Moreover, for each new knot $K_i$, $i=1,2$, we have either:
\begin{enumerate}
    \item\label[casei]{case:loopSurvives}
    all the segments in $\alpha_i$ survive the crushing to become an embedded loop $\ell_i'$ in a component $\tri_i'$ of $\tri'$, so that $(\tri_i',\ell_i')$ forms an edge-ideal triangulation of $K_i$;
    \item\label[casei]{case:nonLoop}
    $\alpha_i$ consists of two type-$1$ segments, and these two segments merge to form a single non-loop edge $\ell_i'$ in a component $\tri_i'$ of $\tri'$; or
    \item\label[casei]{case:loopDeleted}
    $\alpha_i$ is deleted.
\end{enumerate}
\Cref{case:nonLoop,case:loopDeleted} can only happen if $K_i$ is trivial.
If $K_1$ and $K_2$ both fall under \Cref{case:loopSurvives}, then the corresponding components $\tri_1'$ and $\tri_2'$ are distinct.
\end{restatable}

\begin{proof}
Let $\sigma$ be a segment in the ideal loop $\ell$, and let $O_\sigma$ be the induced orbit of $\sigma$.
\begin{itemize}
    \item If $\sigma$ is type-$0$, then $O_{\sigma}$ consists of nothing other than $\sigma$ itself.
    \item If $\sigma$ is type-$2$, then $\sigma$ is the
    only type-$2$ segment amongst the ideal segments.
    Thus, the only ideal segment in $O_{\sigma}$ is $\sigma$ itself.
    \item If $\sigma$ is type-$1$, then there exists exactly one other ideal type-$1$ segment $\sigma'$ appearing on the same side of the $2$-sphere $S$.
    Thus, $\sigma'$ is the only other ideal segment that can appear in $O_{\sigma}$. If $\sigma$ and $\sigma'$ are, indeed, in the same induced orbit, then there must be an embedded disc $\Delta$ built from corner faces with $\partial\Delta = \sigma \cup \sigma' \cup \alpha$, where $\alpha$ is a path of normal arcs in $S$. In particular, $\sigma$ and $\sigma'$ together form a long ideal arc and, because they bound a disc after non-destructively crushing, must represent the unknot. This is the only scenario in which two ideal segments $\sigma$ and $\sigma'$ can be in the same induced orbit.
\end{itemize}

Non-destructively crushing $S$ yields two cell decompositions $D_1^\dagger$ and $D_2^\dagger$, both of the $3$-sphere. The loop $\ell$ is decomposed into two loops $\ell_1^\dagger \subset D_1^\dagger$ and $\ell_2^\dagger \subset D_2^\dagger$,
and $(D_i^\dagger, \ell_i^\dagger)$ forms an ``edge-ideal cell decomposition'' of $K_i$.

Let $D_i^\ast$ denote the cell decomposition obtained by flattening all of the footballs and bigon faces in $D_i^\dagger$; this flattening causes some of the edges in $\ell_i^\dagger$ to be identified, yielding a possibly smaller collection $\ell_i^\ast$ of edges in $D_i^\ast$.

Since $S$ was quad vertex, by \Cref{lem:flattenToEdge}, $D_i^\ast$ is still a $3$-sphere. In addition, $\ell_i^\ast$ is topologically equivalent to $\ell_i^\dagger$ unless two edges in $\ell_i^\dagger$ belong to the same induced orbit.
By the observation made at the beginning of this proof, this only happens if: $\ell_i^\dagger$ is a two-edge unknot, and $\ell_i^\ast$ consists of a single non-loop edge given by identifying the two edges of $\ell_i^\dagger$.

The only remaining non-tetrahedron cells in $D_i^\ast$ are triangular pillows.
In the case where $D_i^\ast$ contains at least one tetrahedron, flattening all of these pillows has no topological effect, and we recover a triangulation $\tri_i'$. Since this flattening does not affect the $1$-skeleton, $\ell_i^\ast$ survives untouched inside $\tri_i'$.
In particular, if $\ell_i^\ast$ forms an embedded closed loop (rather than a non-loop edge), then $(\tri_i',\ell_i^\ast)$ forms an edge-ideal triangulation of $K_i$.

It remains to consider the case where $D_i^\ast$ consists entirely of pillows. In this case $D_i^\ast$ has exactly three edges, and these necessarily form an unknotted loop.
Thus, $\ell_i^\ast$ either:
\begin{itemize}
\item
forms an embedded closed loop, in which case it must form the unknot given by all three edges of $D_i^\ast$; or
\item
consists of a single non-loop edge.
\end{itemize}
Either way, it follows that $K_i$ must form the unknot.
Flattening all triangular pillows deletes $D_i^\ast$ entirely, and with it the trivial knot $K_i$.
\end{proof}

To summarise, \Cref{lem:crushW=0,lem:crushW=2} show that if we crush suitable quad vertex $2$-spheres, then this changes the ideal loop in a topologically meaningful way.
Moreover, as a consequence of the Agol-Hass-Thurston machinery described in \Cref{ssec:orbitSegment}, we can also efficiently keep track of how the ideal loop changes.
These two pieces of machinery are, together, what allow us to give the algorithmic applications that follow in \Cref{sec:algo,sec:np}.

\section{Algorithm and implementation}
\label{sec:algo}

The algorithm below takes a knot $K$ as input, and produces the prime factorisation of $K$ as a collection of edge-ideal triangulations.
As stated, the input is a diagram of $K$, but by skipping \Cref{step:buildIdealEdge} the algorithm can also be seen as taking an edge-ideal triangulation as input.

\begin{algorithm}\label{algo:summands}
Given a diagram of a knot $K$, compute the prime factorisation of $K$ (represented as edge-ideal triangulations) as follows:
\begin{enumerate}
\item\label[stepi]{step:buildIdealEdge}
Build an edge-ideal triangulation $(\tri_K,\ell_K)$ of $K$;
i.e., build a triangulation $\tri_K$ of $\mathbb{S}^3$ such that $K$ is embedded in $\tri_K$ as an ideal loop $\ell_K$. \Cref{ssec:buildEdgeIdeal} describes how this can be done.
\item
Create an empty list $\mathcal{P}$ (which will eventually contain all the prime summands).
Also, create a list $\mathcal{L}$ of edge-ideal triangulations to process, initially containing just $(\tri_K,\ell_K)$.
\item\label[stepi]{step:checkEmpty}
If $\mathcal{L}$ is nonempty, then pick and remove one edge-ideal triangulation $(\tri,\ell)$ from $\mathcal{L}$ and \textbf{continue to \Cref{step:crushKnot}}.
Otherwise, \textbf{terminate with output $\mathcal{P}$}.
\item\label[stepi]{step:crushKnot}
Search for a quad vertex normal $2$-sphere $S$ in $\tri$ with either $w_\ell(S)=0$ or $w_\ell(S)=2$.
	\begin{itemize}
	\item
	If $S$ exists, crush it and keep track of the resulting ideal loops, as discussed in \Cref{sec:identifyingIdealEdges}.
    Append all components of the resulting triangulation to $\mathcal{L}$ that contain an ideal loop. Discard the others.
	\item
	If $S$ does not exist, test whether $\ell$ forms a nontrivial knot. If so, add $(\tri,\ell)$ to $\mathcal{P}$.
	\end{itemize}
In either case, \textbf{return to \Cref{step:checkEmpty}}.
\end{enumerate}
\end{algorithm}

\subsection{Correctness of the algorithm}

\begin{theorem}\label{thm:algCorrect}
\Cref{algo:summands} terminates and correctly factorises the input knot $K$ into prime knots.
\end{theorem}

\begin{proof}
Each step in isolation clearly terminates. Thus, to show that the entire algorithm terminates we just need to
guarantee that \Cref{step:checkEmpty,step:crushKnot} are only repeated finitely many times.
For this, we consider the total number $\tau$ of tetrahedra amongst the triangulations in $\mathcal{L}$.
Each time we reach \Cref{step:checkEmpty} with $\tau > 0$, we remove a triangulation $(\tri,\ell)$ from $\mathcal{L}$, and hence reduce $\tau$ by $|\tri|$.
Although \Cref{step:crushKnot} might increase $\tau$ again by adding some new triangulations back into $\mathcal{L}$,
the new triangulations must have strictly fewer tetrahedra than $\tri$ because they were obtained by crushing a nontrivial normal surface.
Thus, $\tau$ must get strictly smaller each time we repeat \Cref{step:checkEmpty}.
This implies that $\tau$ must eventually reach $0$, at which point the algorithm terminates.

To see that the output is correct, we show that the following properties are satisfied each time the algorithm returns to \Cref{step:checkEmpty}:
\begin{enumerate}[(a)]
\item\label[propi]{property:listL}
Every item in $\mathcal{L}$ is an edge-ideal triangulation of a (possibly trivial, possibly composite) knot.
\item\label[propi]{property:listP}
Every item in $\mathcal{P}$ is an edge-ideal triangulation of a nontrivial prime knot.
\item\label[propi]{property:compose}
The input knot $K$ is given by composing all the knots in $\mathcal{L}\cup\mathcal{P}$.
\end{enumerate}
The first time we reach \Cref{step:checkEmpty}, $\mathcal{P}$ is empty and $\mathcal{L}$ consists of
nothing other than an edge-ideal triangulation of $K$, so these properties are trivially satisfied.
These properties possibly break each time we remove a triangulation from $\mathcal{L}$, so we need to check that \Cref{step:crushKnot} always
adds appropriate triangulations back into either $\mathcal{L}$ or $\mathcal{P}$ to ensure that all three properties are preserved.

We first consider the case where \Cref{step:crushKnot} finds a $2$-sphere with either $w_\ell(S) = 0$ or $w_\ell(S) = 2$.
In this case, we do not modify the list $\mathcal{P}$ at all, so \Cref{property:listP} is clearly preserved.
Moreover, by \Cref{lem:crushW=0,lem:crushW=2}, crushing $S$ produces one of the following results,
and in each case we can easily see that the algorithm preserves \Cref{property:listL,property:compose}:
\begin{description}
\item[Two edge-ideal triangulations $(\tri_1',\ell_1')$ and $(\tri_2',\ell_2')$ such that $\ell=\ell_1'\#\ell_2'$.]
\hfill\\
The algorithm adds both of these edge-ideal triangulations back into $\mathcal{L}$, which ensures that \Cref{property:listL,property:compose} are preserved.
\item[One edge-ideal triangulation $(\tri',\ell')$ such that $\ell'$ is topologically equivalent to $\ell$.]
\hfill\\
The edge-ideal triangulation is potentially accompanied by a single extra $3$-sphere component not containing any ideal loops, but the algorithm ignores any such component.
The algorithm adds $(\tri',\ell')$ back into $\mathcal{L}$, which ensures that \Cref{property:listL,property:compose} are preserved.
\item[No edge-ideal triangulations.]
\hfill\\
There could possibly be up to two $3$-sphere components without ideal loops, but the algorithm ignores these.
This case can only occur if $\ell$ is unknotted, in which case nothing needs to be done to preserve \Cref{property:listL,property:compose}.
\end{description}

On the other hand, if we do not find any suitable $2$-sphere in \Cref{step:crushKnot}, then it follows from \Cref{cor:quadVertexComposite} that $\ell$ must form a prime knot.
If $\ell$ is the unknot, then nothing needs to be done to preserve the three properties.
Otherwise, adding $(\tri,\ell)$ to $\mathcal{P}$ ensures that the three properties are preserved.

By induction, we conclude that when the algorithm terminates, \Cref{property:listL,property:listP,property:compose} must all still hold.
But $\mathcal{L}$ is empty, and hence the edge-ideal triangulations in $\mathcal{P}$ must give the prime factors of $K$, as required.
\end{proof}

\subsection{Optimisations of our implementation}\label{ssec:bottlenecks}

In practice, the following steps are potential bottlenecks for the algorithm:
\begin{description}
\item[\ref{app:bottleneck:buildIdealEdge}.]
Building the edge-ideal triangulation in \Cref{step:buildIdealEdge}.
\item[\ref{app:bottleneck:enumerate}.]
Enumerating quad vertex normal surfaces in \Cref{step:crushKnot}.
\item[\ref{app:bottleneck:nontriviality}.]
Checking nontriviality of the prime knots encountered in \Cref{step:crushKnot}.
\end{description}
Each of these bottlenecks can largely be overcome by augmenting the algorithm with various optimisations.
Our implementation of \Cref{algo:summands}, which includes all of the optimisations that we discuss below, is available at \url{https://github.com/AlexHe98/idealedge}.

The most obvious optimisation is to simplify triangulations before attempting any other computations. If we are working with an ideal triangulation of the knot complement, such as in \ref{app:bottleneck:nontriviality} below,
then we can directly use the simplification tools available in Regina. If we are working with an \emph{edge-ideal} triangulation of the knot, such as in \ref{app:bottleneck:enumerate},
then we must instead use custom simplification tools that are specifically designed to ensure that the ideal loop is preserved;
see \Cref{app:ssec:simplifyEdgeIdeal} for details.

Another set of optimisations is to attempt multiple techniques for solving the same problem in parallel. This is useful because there is often no way to know in advance which technique performs best in any particular case.
Specifically, our implementation uses the following measures.
\begin{enumerate}
\item\label{app:bottleneck:buildIdealEdge}
We use two different methods for turning a knot into an edge-ideal triangulation.
Here, we summarise both methods, and leave the details to \Cref{ssec:buildEdgeIdeal}.
	\begin{itemize}
	\item
	The first method is to directly construct an edge-ideal triangulation with $9n$ tetrahedra, where $n$ is the number of crossings in the knot diagram.
	This approach is guaranteed to terminate, but simplifying the resulting triangulation can be slow due to the relatively large number of tetrahedra.
	\item
	The second method is to triangulate the exterior of the knot, and then attempt to perform a $\frac{1}{0}$ Dehn filling to obtain the desired edge-ideal triangulation.
	Using techniques available in Regina~\cite{Regina}, this is often very fast in practice,
	but the running times are highly variable (indeed, Regina does not even guarantee termination);
	if we add SnapPy~\cite{SnapPy} as an extra dependency, then we can guarantee termination and get more consistent running times.
	Either way, this second method is usually faster than the first one.
	\end{itemize}
\item\label{app:bottleneck:enumerate}
In practice, we find that the quad vertex normal surface enumeration in \Cref{step:crushKnot} often terminates quickly.
This is particularly noticeable when the knot is composite, in which case we typically observe that the algorithm finds a suitable $2$-sphere to crush after enumerating only a handful of quad vertex normal surfaces
(which is usually just a tiny fraction of the total number of such surfaces).
However, occasionally we encounter a triangulation where the enumeration takes a long time;
this is usually because the simplification tools failed to simplify the ideal loop as much as possible.
Our remedy is to run the following in parallel with the enumeration: repeatedly randomise the triangulation, and check whether the new triangulation is one where the enumeration finishes quickly.
See \Cref{ssec:altEnum} for details.
\item\label{app:bottleneck:nontriviality}
We can conclusively determine whether a knot is nontrivial using solid torus recognition, but this can be slow because it relies on normal surface techniques.
In practice, we find that when we run the algorithm on a nontrivial knot, the crushing steps never produce unknotted ideal loops.
Thus, the main opportunity for optimisation here is to use fast heuristics that are usually good at certifying that a knot is nontrivial.
See \Cref{ssec:nontrivialInPractice} for more details.
\end{enumerate}

\subsection{Building an edge-ideal triangulation from a knot}\label{ssec:buildEdgeIdeal}

When the input knot $K$ is presented as a knot diagram, the first task in \Cref{algo:summands} is to convert this into an edge-ideal triangulation of $K$.
We now describe three techniques to achieve this task.
As mentioned in \Cref{ssec:bottlenecks}, our implementation attempts the first two of these techniques in parallel.

\paragraph*{Planar diagram construction}

The first technique takes a knot diagram with $n$ crossings, and guarantees to produce an edge-ideal triangulation of size $9n$ with an ideal loop of length $2n$.
This is excellent in theory since the cost of the conversion is only linear, but as we discuss shortly $9$ is a large enough constant to be a challenge in practice.

To describe the construction, view the input knot diagram as a $4$-regular graph embedded in a $2$-sphere $S$,
with over- and under-crossing information recorded at each vertex of the graph.
The dual graph decomposes $S$ into $4$-sided faces, with one crossing per face.
Viewing the ambient $3$-sphere as a suspension of $S$ (i.e., two cones over $S$, glued together along $S$),
we obtain a decomposition of the $3$-sphere into $4$-sided footballs: each football is a suspension of one of the $4$-sided faces of $S$.
We build an edge-ideal triangulation of the knot by triangulating each football with $9$ tetrahedra,
with two edges corresponding respectively to the over- and under-crossing strands inside that football;
see \Cref{fig:crossingGadget}.

\begin{figure}[htbp]
\centering
	\begin{tikzpicture}
	\node at (0,0) {
		\includegraphics[scale=1]{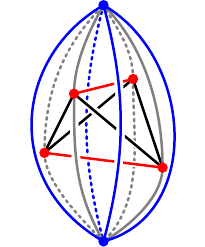}
	};
	\node at (6,0) {
		\includegraphics[scale=1]{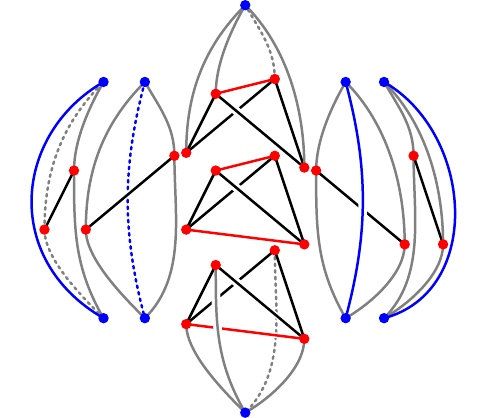}
	};
	\end{tikzpicture}
\caption{The $9$ tetrahedron crossing gadget.
Left: The fully assembled gadget; the two red edges correspond to the over- and under-crossing strands.
Right: Exploded view that shows the individual tetrahedra more clearly.}
\label{fig:crossingGadget}
\end{figure}

Although this construction is fast to perform, the size of the resulting triangulation makes
subsequent computations challenging in practice, so it becomes necessary to simplify the triangulation.
Using combinatorial moves, our implementation guarantees to reduce the length of the ideal loop to $1$, and is usually successful at significantly reducing the size of the triangulation.
However, this simplification is sufficiently slow in practice that it is often a bottleneck for our implementation.

\paragraph*{Dehn filling construction}

To circumvent the bottleneck of simplifying the triangulation given by the first technique,
our second technique seeks to directly construct a very small edge-ideal triangulation.
As we discuss below, if we use only the tools available in Regina~\cite{Regina}, then this
technique involves a heuristic step that does not guarantee to terminate.

Having said this, we \emph{can} guarantee termination if we also use tools from
SnapPy~\cite{SnapPy}. This is not currently included in the \texttt{main} branch of our
implementation, and we do not discuss this in further detail here. Instead, we invite the
interested reader to study the \texttt{snappy} branch of our implementation, where this
feature is included.

Given a knot diagram, our second technique for constructing an edge-ideal triangulation proceeds as follows:
\begin{enumerate}
\item
Build an ideal triangulation of the knot.
There are already longstanding implementations for this in software (such as SnapPy~\cite{SnapPy} and Regina~\cite{Regina}).
\item
Turn the ideal triangulation into a triangulation of the knot exterior with one-vertex boundary.
This can be done by suitably subdividing the tetrahedra so that we can delete a small neighbourhood of the ideal vertex,
and then simplifying the boundary triangulation until it is one-vertex.
Again, Regina~\cite{Regina} already provides the tools to do all this.
\item
Modify the bounded triangulation so that the meridian and longitude of the knot appear as boundary edges.
Regina~\cite{Regina} provides a \emph{heuristic} method for realising this, so it is this step that is not guaranteed to terminate.
\item
To complete the construction, attach a snapped $3$-ball (a gadget built from a single tetrahedron) to ensure that the meridian edge bounds an embedded disc.
Topologically, this realises a $\frac{1}{0}$ Dehn filling, and the longitude edge becomes the ideal loop.
\end{enumerate}

In practice, this heuristic construction is very fast in most cases, and therefore substantially ameliorates the practical disadvantages of the first technique.
However, we do also observe a small proportion of cases where the heuristic is prohibitively slow;
in such cases, the first technique becomes crucial, providing a guaranteed and reasonable upper bound on the running time required to perform the conversion.
By running both techniques in parallel, our implementation thereby gets the best of both worlds.

\paragraph*{$H$-triangulations}

Finally, although we do not use it, we briefly mention a third technique, which follows a method described in  \cite{ibarra2024Htriangulation}.
There, the authors use augmented link diagrams to construct a special ideal triangulation of the knot complement,
which can then be turned into a triangulation of $\mathbb{S}^3$ by adding a single tetrahedron.
The result is referred to as a $H$-triangulation in \cite{aribi2023Htriangulation,ibarra2024Htriangulation},
but can also be interpreted as an edge-ideal triangulation of the knot.

\subsection{Simplifying edge-ideal triangulations}\label{app:ssec:simplifyEdgeIdeal}

Our aim with simplification is to reduce the following (sometimes competing) quantities:
the length of the ideal loop $\ell$ (i.e., the number of edges), the number of vertices in the ambient triangulation $\tri$, and the size of $\tri$ (i.e., the number of tetrahedra).
For reducing the length of $\ell$, we use the following operations:
\begin{description}
\item[Redirecting $\ell$ across a triangle $f$ of the triangulation.]
\hfill\\
This operation is available if $\ell$ has two (consecutive) edges that are both incident to $f$.
In this case, replacing these two edges by the third edge of $f$ reduces the length of $\ell$ by one without changing $\ell$ topologically.
\item[Inserting a snapped ball that effectively collapses an edge of $\ell$.]%
\hspace{-3pt}\footnote{Inserting a snapped ball is dual to Matveev's arch construction for special spines;
see~\cite[Section 1.1.4]{matveev2007algorithmic}.}%
\hfill\\
This operation is available along any internal edge $e$ connecting two distinct vertices of $\tri$
(so, in particular, it is available along any edge of $\ell$ when $\ell$ has length greater than one).
Fix any particular triangle $f$ incident to $e$, and let $v_0$ and $v_1$ denote the vertices at the endpoints of $e$.
Also, let $\varepsilon_0$ and $\varepsilon_1$ denote the two edges of $f$ other than $e$, labelled so that $\varepsilon_i$ is incident to $v_i$.
The operation proceeds by undoing the face-gluing along $f$, and inserting a snapped $3$-ball in such a way that $\varepsilon_0$ and $\varepsilon_1$ become identified
(see \Cref{fig:snapEdge});
since the endpoints of $e$ are distinct, the overall operation does not change the ambient $3$-manifold.
Combinatorially, this operation reduces the number of vertices by one, but at the cost of increasing the size of $\tri$ by one.

There is a way to realise this operation as a composition of two other well-known operations.
The first step of the composition is to pinch the endpoints of $e$ together by inserting Jeff Weeks' two-tetrahedron triangular-pillow-with-tunnel
(in Regina, this is implemented as the \texttt{pinchEdge()} routine in the \texttt{Triangulation<3>} class).
The inserted triangular pillow always admits a 2-1 edge move that yields the same overall result as inserting a snapped ball in the manner described above.
\end{description}

\begin{figure}[htbp]
\centering
	\includegraphics[scale=1]{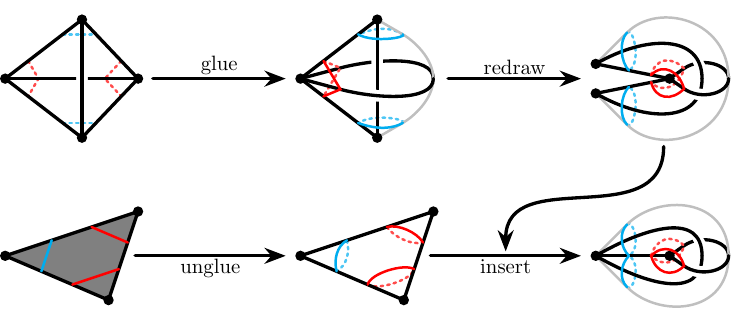}
\caption{Shortening the ideal loop: This is equivalent to performing Jeff Week's insertion of a triangular pillow with tunnel, followed by a 2-1 edge move. The ideal edge is the rightmost edge on the shaded triangle at the bottom. Endpoints of this edge must be distinct.}
\label{fig:snapEdge}
\end{figure}

As described above, we can insert snapped balls to identify any two distinct vertices that are joined together by an edge;
thus, we can (and do) also use this operation to reduce $\tri$ to a one-vertex triangulation (but, again, at the cost of increasing the size of $\tri$).
To reduce the size of $\tri$, we use a combination of the following standard moves for triangulations:
3-2 moves, 2-0 edge moves, 2-1 edge moves, and 4-4 moves.
We require bespoke implementations of these triangulation moves to keep track of the location of the ideal loop after performing each move.
The order in which the moves are performed is inspired by the simplification routines from both SnapPea \cite{SnapPy} and Regina \cite{Regina}.
The interested reader is invited to peruse the source code to see how this works in detail.

\subsection{Randomisation and alternate enumerations}\label{ssec:altEnum}

Empirically, we observe that simplifying an edge-ideal triangulation of a composite knot typically leads to
a triangulation in which enumerating quad vertex normal surfaces quickly finds a decomposing $2$-sphere.
When we get unlucky and encounter a triangulation in which the enumeration does not finish quickly,
we find that it is often helpful to run an alternate enumeration on a different triangulation.

In detail, to search for quad vertex $2$-spheres, as specified in \Cref{step:crushKnot} of \Cref{algo:summands}, our implementation runs the following three procedures in parallel:
\begin{itemize}
\item
Given an edge-ideal triangulation $(\tri,\ell)$, search for a $2$-sphere $S$ with either
$w_\ell(S)=0$ or $w_\ell(S)=2$ by enumerating quad vertex normal surfaces.
This main enumeration is guaranteed to terminate, but can take a long time.
\item
Repeatedly randomise $(\tri,\ell)$ by first using random 2-3 moves to significantly increase the number of tetrahedra, and then trying to simplify again.
For each randomised triangulation $(\tri',\ell')$, if $\tri'$ has at most as many tetrahedra as the original triangulation $\tri$,
then send $(\tri',\ell')$ to a parallel process running alternate enumerations (see below).
\item
Each time we receive a new randomised triangulation $(\tri',\ell')$, terminate the current alternate enumeration (if there is one),
and start a new alternate quad vertex normal surface enumeration on $\tri'$.
If any of the alternate enumerations ever finds a suitable $2$-sphere $S$, then we can cancel the main enumeration early and proceed using $S$.
\end{itemize}
In practice, whenever the main enumeration happens to take a long time, we usually find that after only a handful of randomisation attempts,
one of the alternate enumerations is able to find a suitable $2$-sphere $S$.

\subsection{Detecting nontrivial knots in practice}\label{ssec:nontrivialInPractice}

In \Cref{step:crushKnot}, \Cref{algo:summands} needs to decide whether each prime knot it finds is nontrivial.
This can always be done using solid torus recognition.
Although Regina's implementation of solid torus recognition~\cite{Regina,burton2012unknot} is heavily optimised, it can nevertheless be a bottleneck.

Experiments suggest that, provided we simplify our triangulations, we never crush a $2$-sphere that splits off a trivial component.
In other words, if the initial input knot is nontrivial, then in practice we expect that we \emph{only} encounter edge-ideal triangulations of nontrivial knots.
Therefore, in practice, we should focus on getting fast verification of nontriviality.
Our implementation does this using the following tests:
\begin{itemize}
\item
We can pinch the ideal loop down to an ideal vertex, and attempt to certify hyperbolicity (and hence nontriviality) using strict angle structures;
this is very fast because it can be done purely using linear programming.
Even better, since hyperbolic knots are always prime, we can perform this check \emph{before}
attempting to use quad vertex normal surfaces to certify primeness.
Also, to avoid indiscriminately searching for strict angle structures on every single knot that the algorithm encounters,
we can use SnapPea to identify good candidate knots to check for hyperbolicity.
\item
We can try to certify that the fundamental group $\pi_1$ is not $\mathbb{Z}$ by enumerating transitive permutation representations --
in other words, enumerating transitive group actions of $\pi_1$ on a $k$-element set, for some $k \geqslant 2$.
For $\pi_1 \simeq \mathbb{Z}$, there is (up to conjugacy) only one such action for each $k$, and the stabiliser is isomorphic to $\mathbb{Z}$.
Thus, if we find that there is more than one transitive action, or we can find a transitive action whose stabiliser is not isomorphic to $\mathbb{Z}$
(which Regina is often able to certify by verifying that the \emph{abelianisation} of the stabiliser is not isomorphic to $\mathbb{Z}$),
then we know that $\pi_1$ cannot be $\mathbb{Z}$, and hence the knot must be nontrivial.
Regina is able to enumerate transitive actions for each $k \leqslant 7$, and this is very fast for $k \leqslant 6$.
Thus, our implementation always performs the enumerations for $k \leqslant 6$ before attempting solid torus recognition;
if this fails, then we run solid torus recognition in parallel with the enumeration for $k = 7$.
\end{itemize}

\subsection{Experimental results}\label{ssec:experiment}

In this section we demonstrate that \Cref{algo:summands}
is practical -- and in some cases even quite efficient -- for typical inputs.
We do this by testing our implementation on knots obtained from a range of sources.
The code and the experimental data can be accessed at \url{https://github.com/AlexHe98/idealedge}.

Although the design of our algorithm is best-suited for composite knots, we also test the performance on prime knots.
There are two reasons for this:
\begin{itemize}
\item
First, if we are interested not just in deciding whether a knot is composite, but also in obtaining the full prime factorisation, then our algorithm still needs to certify that each factor is prime.
Therefore, the performance in the prime case is relevant even if the input knot is composite.
\item
Second, empirical experience in other settings suggests that working with quad (rather than standard) coordinates is often crucial for practical performance.
This is relevant in our setting regardless of whether the input knot is prime or composite.
\end{itemize}
All of our experiments were run on a laptop with an Intel Core i5-7200U processor.

For prime knots, \Cref{algo:summands} certified primeness of:
\begin{enumerate}[(a)]
\item
all $122$ torus knots with $15$ to $100$ crossings in $310.02$ seconds ($2.54$ seconds per knot)
\item
all $376$ satellite knots from the census~\cite{burton2020knots} of prime knots with $15$ to $19$ crossings in $404.66$ seconds ($1.08$ seconds per knot)
\item
a random sample of $100$ hyperbolic knots with $15$ or $16$ crossings in $138.48$ seconds ($1.38$ seconds per knot)
\end{enumerate}
Note that for hyperbolic knots, efficient heuristics for verifying hyperbolicity essentially always succeed in practice, and so the crushing part of our algorithm is rarely called upon in such cases.

For composite knots, we randomly generated knot diagrams in two ways.
First, we randomly sampled prime knots from the census, and used the standard construction to form composites with up to $8$ prime summands and $152$ crossings.
Timings for these composite knots are summarised in the table below.

\begin{center}
\begin{tabular}{c|c|c}
  $\#$ prime summands & $\#$ of samples & avg. running time per knot \\
  \hline
  \hline
  $2$ & $100$ & $4.71$ \\
  \hline
  $3$ & $100$ & $10.95$  \\
  \hline
  $4$ & $100$ & $24.06$ \\
  \hline
  $5$ & $100$ & $33.06$ \\
  \hline
  $6$ & $10$ & $47.75$ \\
  \hline
  $7$ & $10$ & $70.21$ \\
  \hline
  $8$ & $10$ & $88.68$ 
\end{tabular}
\end{center}

Second, we generated $30$ composite knots with two $15$-crossing prime summands, but using
a non-standard construction yielding diagrams that ``look prime''. More specifically, the
construction ``overlays'' braid representations of the two prime summands and then uses
SnapPy~\cite{SnapPy} to simplify the resulting knot diagram (see the source code for
details). Although SnapPy often succeeds in sufficiently simplifying the diagram to see
that the knot is composite, simplifications fail often enough to yield a reasonable
(though somewhat slow) method to generate hard diagrams of composite knots. For the $30$
knot diagrams generated this way, \Cref{algo:summands} factorised all of them in $124.82$
seconds ($4.16$ seconds per knot). Strikingly, these running times are comparable to the
running times using standard diagrams for composite knots as input. This suggests that
primeness of the input diagram only has a limited effect on running times after converting
the input into an edge-ideal triangulation.

In addition to the experiments on knot diagrams, we also used the random-walk algorithm
from~\cite{AltmannMCMC} to sample one-vertex triangulations of $3$-spheres, and look at
the knot types of their edges. Since unknotted edges are extremely common, we increase the
variety of knots encountered by starting the random walk at a triangulation with no
unknotted edges from~\cite{burton2023counterexamples}. We sampled $285$ knots in
triangulations of size $25$--$30$.
Our implementation only needed $65.28$ seconds to compute the prime factorisations of all $285$ of these knots ($0.23$ seconds per knot).
These
small running times are not surprising, since \emph{(a)} the sampled triangulations are
relatively small, and \emph{(b)} we avoid the (usually nontrivial) work of constructing
edge-ideal triangulations in this setting.

\section{\texorpdfstring{\NP}{NP} certificate for composite knots}\label{sec:np}

We now adapt our techniques to prove \Cref{thm:summandsNP}:
\Cref{problem:summands} is in \NP, and hence \textsc{Composite Knot Recognition} is in \NP.
Recall that \Cref{problem:summands} asks whether a given knot $K$ has at least $k$ summands in its prime factorisation.
At a high level, our certificate is very natural:
it consists of a sequence of $2$-spheres that decompose $K$ (via crushing) into
$k$ nontrivial summands, together with certificates that each of the $k$ summands is indeed nontrivial.

In detail, suppose we are given an $n$-crossing diagram of a knot $K$.
If the prime factorisation of $K$ has at least $k$ summands, then we claim that we can
verify this fact in polynomial time in $n$ using a certificate consisting of the following data:
\begin{itemize}
\item
A sequence $(\tri_0,L_0), \ldots, (\tri_m,L_m)$ of edge-ideal triangulations, satisfying the following invariant:
for each $i\in \{0,\ldots,m\}$, $(\tri_i,L_i)$ is a union of edge-ideal triangulations of knots in $\mathbb{S}^3$,
such that the input knot $K$ can be obtained by composing all of these knots.
\item
For each $i \in \{0, \ldots, m-1\}$, a component $(C_i,\ell_i)$ of $(\tri_i,L_i)$, together with a quad vertex normal
$2$-sphere $S_i$ in $C_i$ such that either $w_{\ell_i}(S_i) = 0$ or $w_{\ell_i}(S_i) = 2$.
\item $k$ components $\ell_{m_0}, \ldots, \ell_{m_{k-1}}$ of $L_m$, together with certificates that each $\ell_{m_i}$ forms a nontrivial knot in its component.
\end{itemize}

The key idea is that the invariant can be efficiently verified in an iterative fashion.
Then, after verifying that the invariant holds for $L_m$, all that remains is to verify that $L_m$ has $k$ nontrivial components;
by uniqueness of prime factorisation, this is enough to verify that the input knot $K$ has at least $k$ summands.
We give the full proof details in \Cref{ssec:proofNP}, but for now we simply outline the polynomial-time verification algorithm:
\begin{enumerate}
\item\label[stepi]{step:buildEdgeIdealNP}
Follow \Cref{ssec:buildEdgeIdeal} to build an edge-ideal triangulation $(\tri_0',L_0')$ of $K$,
and then verify that $(\tri_0',L_0')$ is combinatorially isomorphic to $(\tri_0,L_0)$. Note that checking combinatorial isomorphy of triangulations is a polynomial-time procedure
(see, e.g., \cite[Lemma~3.1 and Remark~3.2]{schleimer2011sphere}).
\item\label[stepi]{step:quadVertexNP}
For each $i \in \{0, \ldots, m-1\}$, check that either $w_{\ell_i}(S_i) = 0$ or $w_{\ell_i}(S_i) = 2$.
Also, follow \Cref{ssec:certifyQuadVertexSphere} to verify that $S_i$ is a quad vertex $2$-sphere.
\item\label[stepi]{step:crushNP}
For each $i \in \{0, \ldots, m-1\}$, crush $S_i$ to obtain a new triangulation $\tri_{i+1}'$, and then follow \Cref{sec:identifyingIdealEdges}
(recall that this uses work of Agol, Hass and Thurston~\cite{agol2006computational}) to identify the new collection $L_{i+1}'$ of ideal loops.
Verify that $(\tri_{i+1}',L_{i+1}')$ is combinatorially isomorphic to $(\tri_{i+1},L_{i+1})$.
\item\label[stepi]{step:nontrivialNP}
Follow \Cref{ssec:certifyNontrivial} to verify that each of $\ell_{m_0}, \ldots, \ell_{m_{k-1}}$ forms a nontrivial knot.
\end{enumerate}

\subsection{Certifying quad vertex \texorpdfstring{$2$}{2}-spheres}\label{ssec:certifyQuadVertexSphere}

Since our \NP\ certificate involves a sequence of quad vertex $2$-spheres that need to be crushed,
the verifier needs to be convinced that each $2$-sphere is actually quad vertex, and hence that the lemmas in \Cref{ssec:crushIdealLoop} apply.
Here, we describe how to verify in polynomial time that a given normal surface $S$ is a quad vertex $2$-sphere.

First, we can easily check that $\chi(S) = 2$.
We can also explicitly construct the induced orbits of type-$1$ segments to verify that they are trivial.
What remains is to verify that $S$ is connected, and that the induced orbits of type-$2$ segments are also trivial.
We claim that we can check both of these at once by simply counting the number of components of the \textbf{guts} --
i.e., the components given by cutting along $S$ and then deleting all the parallel cells;
note that it is straightforward to count these components in polynomial time, because a normal surface in a $t$-tetrahedron triangulation induces at most $6t$ non-parallel cells.

To see that this count suffices, first observe that if $S$ is a quad vertex $2$-sphere, then the guts must have exactly two components:
cutting along $S$ yields two components (because every $2$-sphere in $\mathbb{S}^3$ is separating), and then deleting the parallel cells does not change
the number of components because the induced orbits of the type-$2$ segments are all trivial.

Conversely, suppose the guts have exactly two components.
We claim that $S$ must be a $2$-sphere, and that all the type-$2$ segments must have trivial induced orbits.
To see why, let $c$ be the number of normal isotopy classes amongst the components of $S$.
Since the $3$-sphere does not contain any non-separating surfaces, observe that after
cutting along $S$, exactly $c+1$ of the resulting components contain at least one non-parallel cell.
Thus, after deleting all the parallel cells, we see that the guts must have at least $c+1$ components.
Since $S$ is nonempty, we therefore conclude that $c=1$;
in other words, the components of $S$ (if there is more than one) must all be normally isotopic to each other.
But we already know that $\chi(S)=2$, so $S$ must consist entirely of a single $2$-sphere component.
Finally, observe that if any type-$2$ segment has nontrivial induced orbit, then deleting the parallel cells in this orbit disconnects one of the components given by cutting along $S$, which forces the guts to have more than two components.
Thus, $S$ must be a quad vertex $2$-sphere.

\subsection{Certifying nontrivial ideal loops}\label{ssec:certifyNontrivial}

We use Lackenby's work from \cite{lackenby2016efficient} to certify knottedness of nontrivial ideal loops. We can do this in one of two possible ways.
\begin{enumerate}
    \item The seemingly more straightforward approach is to use the certificate for knottedness in \cite[Theorems~1.1 and~1.5]{lackenby2016efficient},
    which can be verified in time polynomial in the number of crossings in the input knot diagram.
    However, we cannot use this result directly, since we cannot easily and efficiently convert our edge-ideal triangulations into knot diagrams.
    Fortunately, the very first step in \cite{lackenby2016efficient} is to build a triangulation of the knot exterior with number of tetrahedra linear in the crossing number of the input diagram.
    Hence, the certificate from \cite{lackenby2016efficient} can also be verified in time polynomial in the size of a triangulation of the knot exterior.
    Converting an edge-ideal triangulation to a triangulation of the knot exterior is straightforward,
    but some care is needed to represent the homological longitude of the knot exterior, which is part of the input of \cite[Theorem~1.5]{lackenby2016efficient}.
    \item Alternatively, we can use \cite[Theorem~1.7]{lackenby2016efficient} out of the box.
    This theorem states that checking whether a $3$-manifold has incompressible boundary is in {\NP}.
    For knot exteriors, checking this condition is equivalent to checking whether the knot is nontrivial.
\end{enumerate}

\subsection{Correctness of the \NP\ certificate}\label{ssec:proofNP}

We now prove that if $K$ has least $k$ summands then a suitable certificate always exists (and hence the verification algorithm succeeds),
and conversely that if the verification algorithm succeeds then $K$ must have at least $k$ summands.

As mentioned earlier, the key is that the sequence $(\tri_0,L_0),\ldots,(\tri_m,L_m)$ from the certificate must satisfy the following invariant:
for each $i\in \{0,\ldots,m\}$, the input knot $K$ can be obtained by composing all the knots given by the components of $(\tri_i,L_i)$.
This invariant is essentially a consequence of \Cref{lem:crushW=0,lem:crushW=2}.
In detail, recall that for each $i\in \{0,\ldots,m-1\}$, the certificate specifies a particular component $(C_i,\ell_i)$ of $(\tri_i,L_i)$,
together with a quad vertex $2$-sphere $S_i$ in $C_i$ such that either $w_{\ell_i}(S_i)=0$ or $w_{\ell_i}(S_i)=2$.
By \Cref{lem:crushW=0,lem:crushW=2}, crushing $S_i$ replaces $(C_i,\ell_i)$ with either:
a new edge-ideal triangulation of the same knot, or a pair of edge-ideal triangulations giving a pair of knots obtained by a (possibly trivial) connected sum decomposition of $\ell_i$.
Thus, if the invariant holds for $L_i$, then it also holds for $L_{i+1}$.

With this inductive argument in mind, suppose the prime factorisation of $K$ has at least $k$ summands.
We produce the required certificate via the following procedure:
\begin{enumerate}
\item
Following the construction from \Cref{ssec:buildEdgeIdeal}, build an edge-ideal triangulation $(\tri_0,L_0)$ of $K$, so that the invariant holds trivially for $L_0$.
Note that $\tri_0$ has $9n$ tetrahedra, and $L_0$ consists of a single ideal loop $\ell_0$ of length $2n$.
\item
Set $i = 0$.
(The idea is to inductively construct edge-ideal triangulations $(\tri_i,L_i)$ until we reach some
$m \geqslant 0$ such that among the components of $(\tri_m,L_m)$, we have $k$ ideal loops forming nontrivial knots.)
\item\label[stepi]{step:certificateLoop}
If $L_i$ includes at least $k$ nontrivial components, then use \Cref{ssec:certifyNontrivial} to obtain certificates of nontriviality;
this completes the construction of the certificate, so set $m = i$ and \textbf{terminate}.
Otherwise, since the invariant holds for $(\tri_i,L_i)$, and since the prime factorisation of $K$ is unique, the components of $(\tri_i,L_i)$ cannot all be prime knots.
Thus, we may fix a component $(C_i,\ell_i)$ giving a composite knot.
\item
Since $\ell_i$ is composite, \Cref{lem:nontrivialS2} gives a quad vertex normal $2$-sphere $S_i$ such that either $w_{\ell_i}(S_i) = 0$ or $w_{\ell_i}(S_i) = 2$.
Because $S_i$ is quad vertex, the number of bits needed to encode its coordinates is polynomial in $n$.
\item
Crush $S_i$ to obtain a new edge-ideal triangulation $(\tri_{i+1},L_{i+1})$.
By the inductive argument above, the invariant holds for $(\tri_{i+1},L_{i+1})$.
Note also that $|\tri_{i+1}| < |\tri_i|$ and $|L_{i+1}| \leqslant |L_i| + 2$, where $|L_i|$ is the number of edges in $L_i$.
Increase $i$ by $1$ and \textbf{return to \Cref{step:certificateLoop}}.
\end{enumerate}
Observe that the verification algorithm will succeed if it is passed a certificate constructed in this way.

For the converse, suppose the verification algorithm succeeds.
We use the invariant to show that the prime factorisation of $K$ must have at least $k$ summands.
Since \Cref{step:buildEdgeIdealNP} succeeds, $(\tri_0,L_0)$ is itself an edge-ideal triangulation of $K$, and hence the invariant holds for $i=0$.
Since \Cref{step:quadVertexNP,step:crushNP} succeed, the above inductive argument applies, and hence the invariant holds for all $i\in \{0,\ldots,m\}$.
\Cref{step:nontrivialNP} then verifies that components of $(\tri_m,L_m)$ include $k$ nontrivial knots.
By uniqueness of the prime factorisation of $K$, this therefore verifies that $K$ has at least $k$ summands.

\section{Triangulation complexity}\label{sec:further}

In this section we leverage the fact that crushing a nontrivial normal surface reduces the size of its ambient triangulation to study \textbf{triangulation complexity} --
i.e., the minimum number of tetrahedra necessary to triangulate a given manifold within a fixed class of triangulations.

\subsection{Triangulation complexity of edge-ideal triangulations.}

Let $\mathcal{M}$ be the exterior of a knot $K$ in $\mathbb{S}^3$ and let $\tilde{c}(\mathcal{M})$ be the triangulation complexity of $\mathcal{M}$ over the class of edge-ideal triangulations.
We begin with an observation in the case where $K$ is prime:

\begin{lemma}\label{lem:minEdgeIdealPrime}
Let $K$ be a prime knot. If $(\tri,\ell)$ is an edge-ideal triangulation of $K$ that
realises the edge-ideal complexity of $\manifold = \mathbb{S}^3\setminus K$ (i.e.,
$|\tri| = \tilde{c}(\mathcal{M})$), then the ideal loop $\ell$ must have length one.
\end{lemma}

\begin{proof}
First, if $K$ is the unknot, then one can check from the census of closed $3$-manifold triangulations that $\tilde{c}(\mathbb{S}^3\setminus K) = 1$.
In fact, in both of the one-tetrahedron triangulations of $\mathbb{S}^3$, all of the unknotted ideal loops have length one, and so the lemma holds for the unknot.
Thus, for the rest of the proof, we assume that $K$ is a nontrivial prime knot.

Let $(\tri,\ell)$ be an edge-ideal triangulation of $K$, and suppose $\ell$ has length greater than one.
Fix an edge $e$ in $\ell$, and note that the endpoints of $e$ must be distinct.
Thus, a small regular neighbourhood of $e$ is bounded by an embedded $2$-sphere $S$ with $w_\ell(S)=2$.
We claim that we can turn $S$ into a nontrivial normal $2$-sphere $S'$ with $w_\ell(S')=2$.
The result then follows easily from this claim:
\Cref{lem:quadVertex} gives a quad vertex $2$-sphere $Q$ such that either $w_\ell(Q)=0$ or $w_\ell(Q)=2$,
and so \Cref{lem:crushW=0,lem:crushW=2} tell us that crushing $Q$ yields a new edge-ideal triangulation of $K$ with fewer tetrahedra than $\tri$.

All that remains is to justify the claim.
We begin by introducing some helpful ad hoc terminology.
Call any embedded $2$-sphere $F$ in $\tri$ a \textbf{successor} of $S$ if:
\begin{itemize}
\item
$F$ is disjoint from the ideal edge $e$; and
\item
$F$ can be related to $S$ via an ambient isotopy fixing the ideal loop $\ell$ (as a set).
\end{itemize}
Note that such a successor $F$ must intersect $\ell$ in exactly two points;
call these points the \textbf{poles} of $F$, and call any arc in $F$ joining one pole to the other a \textbf{longitude} of $F$.
The poles subdivide $\ell$ into two arcs:
an \textbf{inner} arc containing the ideal edge $e$, and an \textbf{outer} arc.
Recalling that $S$ is the $2$-sphere bounding a small regular neighbourhood of $e$,
we use the defining ambient isotopy between $F$ and $S$ to see that the inner arc cobounds an embedded disc with any longitude of $F$.
This implies that the outer arc \emph{cannot} cobound a disc with any longitude of $F$ (using Dehn's lemma, together with the fact that the ideal loop $\ell$ gives a nontrivial knot).

With all this in mind, it suffices to find a nontrivial normal $2$-sphere $S'$ that is a successor of $S$.
We do this by normalising $S$, and explaining why the resulting normal surface must have a component that is a successor of $S$.
We argue inductively, using the fact that the normalisation of $S$ can be broken down into a finite sequence of the following operations:
\begin{itemize}
\item
surgery along a (necessarily inessential) compression disc disjoint from the $1$-skeleton of~$\tri$;
\item
ambient isotopy that reduces the weight of an edge; and
\item
deletion of $2$-sphere components disjoint from the $1$-skeleton of $\tri$.
\end{itemize}
The deletions of $2$-sphere components clearly preserve the existence of a successor (since a successor intersects the $1$-skeleton at the poles).
The surgeries split one ``old'' $2$-sphere component into two ``new'' $2$-sphere components;
if the old component was a successor, then observe that one of the new components must be a successor.

Finally, the only way the ambient isotopies could fail to preserve a successor is if they isotoped the successor across the ideal loop $\ell$ (thereby removing the successor's points of intersection with $\ell$).
But normalisation can never even temporarily increase the weight on edge $e$ (or indeed any edge of $\tri$), and so we can never isotope across the inner arc.
We can never isotope across the outer arc either, because no longitude cobounds an embedded disc with the outer arc (as observed earlier).
Therefore, the ambient isotopies involved in the normalisation procedure must always preserve the existence of a successor.
This completes the proof.
\end{proof}

Suppose now that $K$ is composite.
Write $K = K_1 \# \cdots \# K_m$ for the unique prime factorisation of $K$ (hence each $K_i$, $1\leqslant i \leqslant m$, is nontrivial and prime).
Starting with an edge-ideal triangulation $(\tri,\ell)$ for $K$, our results from \Cref{ssec:crushIdealLoop} ensure that we can iteratively find quad vertex spheres to crush at least $m-1$ times to obtain edge-ideal triangulations for each of the prime knots $K_1,\ldots,K_m$.
We show that if we only crush $m-1$ times, then one additional crush is always possible, and so in fact it is always possible to crush at least $m$ times in total.

To see this, note that if we only crush $m-1$ times, then every crush performs a nontrivial connected sum decomposition.
So after each crush, by \Cref{lem:crushW=2}, every ideal segment survives to form a new ideal edge.
Since each crushed sphere intersects the ideal loop twice, the number of ideal segments is exactly $2$ more than the number of ideal edges, and hence the number of ideal edges increases by exactly $2$ after each crush.
Thus, we end up with $m$ prime knots represented by a total of at least $2m-1$ ideal edges;
in other words, we have an ``excess'' of $m-1$ ideal edges.
Consequently, at least one of the prime knots $K_i$ is represented by an ideal loop $\ell_i$ of length greater than one,
so following the proof of \Cref{lem:minEdgeIdealPrime} gives the additional crush that we require.

Since we strictly decrease the overall number of tetrahedra after each crush, crushing $m$ times establishes the following result:

\begin{proposition}\label{prop:edgeIdealBound}
For any composite knot $K$, with unique prime factorisation $K = K_1 \# \cdots \# K_m$, we have
\[
\tilde{c}(\mathbb{S}^3 \setminus K) \geqslant m + \sum_{i=1}^{m} \tilde{c}(\mathbb{S}^3 \setminus K_i).
\]
\end{proposition}

Our proof of \Cref{prop:edgeIdealBound} assumes that any additional sphere we can crush eliminates all excess ideal edges at once.
However, it might be possible to continue crushing spheres in a more controlled way -- for instance, so that each crush only removes one excess ideal edge -- and thereby obtain a stronger bound.
This makes no difference when $m=2$, but even in this case proving tightness of the bound in \Cref{prop:edgeIdealBound} is not straightforward, as demonstrated by the following example.

\begin{example}
  \label{appex:bounds}
  From the census of small triangulations of $\mathbb{S}^3$ we see that there are two knot complements of ideal edge complexity $1$: the unknot and the trefoil $3_1$ (both are represented by edges in the unique one-tetrahedron one-vertex $3$-sphere triangulation).
The only knot complement of edge-ideal complexity $2$ is $5_1$.
The knot complements of edge-ideal complexity $3$ are $4_1$, $7_1$, $8_{19}$ and $10_{124}$.

  It follows from \Cref{prop:edgeIdealBound} that the edge-ideal complexity of the complements of $3_1 \# 3_1$ and $3_1 \# 5_1$ are at least $4$ and $5$, respectively.
  However, such triangulations of $\mathbb{S}^3$ do not exist in the census.
Instead, the smallest triangulations of $\mathbb{S}^3$ containing an ideal edge of type $3_1 \# 3_1$ and $3_1 \# 5_1$ have sizes $5$ and $6$, respectively.
\end{example}

\Cref{appex:bounds} motivates the following question.

\begin{question}
    Is there a composite knot that realises tightness for \Cref{prop:edgeIdealBound}?
\end{question}

\subsection{Triangulation complexity of ideal triangulations.}

We now adapt our observations about edge-ideal triangulations to obtain results about the complexity of ideal (in the usual sense) triangulations of knot complements.

Given an edge-ideal triangulation $(\tri,\ell)$ of $K$ with $\ell$ consisting of $k$ edges, we can obtain an ideal triangulation of the complement of $K$ by pinching all of these edges.
Moreover, this can always be done at the cost of adding just $k+1$ additional tetrahedra.
In detail, we can first reduce the length of the ideal loop to one by inserting $k-1$ snapped balls, as described in \Cref{app:ssec:simplifyEdgeIdeal}.
Then, to pinch the entire ideal loop down to a single ideal vertex, we insert Jeff Weeks' two-tetrahedron triangular-pillow-with-tunnel
(in Regina, this can be done by calling the \texttt{pinchEdge()} routine from the \texttt{Triangulation<3>} class).

When $K$ is prime, recall that the edge-ideal complexity of $K$ is always realised by an edge-ideal triangulation whose ideal loop has length one.
As just described, we can turn this into an ideal triangulation by inserting two extra tetrahedra, so we have the following:

\begin{corollary}\label{cor:idealBound}
    Let $K$ be a prime knot. Then
    \[
        \hat{c}(\mathbb{S}^3 \setminus K) \leqslant \tilde{c}(\mathbb{S}^3 \setminus K) + 2
    \]
    where $\hat{c}(M)$ denotes the triangulation complexity over the class of ideal triangulations.
\end{corollary}

Moreover, when $K$ is composite, applying \Cref{cor:idealBound} to each prime summand $K_i$ in \Cref{prop:edgeIdealBound} yields an additional lower bound
$\tilde{c}(\mathbb{S}^3 \setminus K) \geqslant -m + \sum_{i=1}^{m} \hat{c}(\mathbb{S}^3 \setminus K_i)$.

\bibliography{references}

\end{document}